\documentclass[12pt]{amsart}
\usepackage{amssymb}
\usepackage{amsthm}
\usepackage{enumerate}
\usepackage[dvips]{graphics,color}

\theoremstyle{plain}
\newtheorem{Thm}{Theorem}[section]
\newtheorem{Lem}[Thm]{Lemma}
\newtheorem{Prop}[Thm]{Proposition}
\newtheorem{Cor}[Thm]{Corollary}
\newtheorem{remark}[Thm]{Remark}
\newtheorem{question}[Thm]{Question}
\newtheorem{example}[Thm]{Example}

\theoremstyle{remark}

\small\normalsize
\numberwithin{equation}{section}

\newcommand\less{\leqslant}
\newcommand\great{\geqslant}

\newcommand\Ima{\mathop{\rm Im}\nolimits}
\newcommand\clos{\mathop{\rm Clos}\nolimits}

\begin{document}

\sloppy

\title[Weighted norm inequalities]
{Weighted norm inequalities
for \\ de Branges--Rovnyak  spaces
and their applications}
\author{Anton Baranov, Emmanuel Fricain, Javad Mashreghi}

\address {Department of Mathematics and Mechanics, St. Petersburg State
University, 28, Universitetskii pr., St. Petersburg, 198504, Russia}
\email{A.Baranov@ev13934.spb.edu}

\address{Universit\'e de Lyon; Universit\'e Lyon 1; Institut Camille Jordan
CNRS UMR 5208; 43, boulevard du 11 Novembre 1918, F-69622 Villeurbanne }
\email{fricain@math.univ-lyon1.fr}

\address{D\'epartement de math\'ematiques et de statistique,
         Universit\'e Laval,
         Qu\'ebec, QC,
         Canada G1K 7P4.}
\email{Javad.Mashreghi@mat.ulaval.ca}

\thanks{This work was supported by funds from NSERC (Canada),
Jacques Cartier Center (France) and RFBR (Russia). }

\begin{abstract}
Let $\mathcal{H}(b)$ denote the de Branges--Rovnyak space associated with
a function $b$ in the unit ball of $H^\infty(\mathbb{C}_+)$. We study the
boundary behavior of the derivatives of functions in $\mathcal{H}(b)$ and
obtain weighted norm estimates of the form $\|f^{(n)}\|_{L^2(\mu)} \le
C\|f\|_{\mathcal{H}(b)}$, where  $f \in \mathcal{H}(b)$ and $\mu$ is
a Carleson-type measure on $\mathbb{C}_+\cup\mathbb{R}$. We
provide several applications of these inequalities.
We apply them to obtain embedding theorems for $\mathcal{H}(b)$ spaces.
These results extend Cohn and Volberg--Treil embedding theorems for the
model (star-invariant) subspaces which are special classes of de
Branges--Rovnyak spaces. We also exploit the
inequalities for the derivatives to study stability of Riesz bases
of reproducing kernels $\{k^b_{\lambda_n}\}$ in $\mathcal{H}(b)$ under small
perturbations of the points $\lambda_n$.

\end{abstract}

\maketitle

\section{Introduction}
Let $\mathbb{C}_+$ denote the upper half-plane in the
complex plane and let
$H^2(\mathbb{C}_+)$ denote the usual Hardy space on $\mathbb{C}_+$.
For $\varphi\in L^\infty(\mathbb{R})$, let $T_\varphi$ stand for the
Toeplitz operator defined on $H^2(\mathbb{C}_+)$ by
$$
T_\varphi f:=P_+(\varphi f),\qquad f\in H^2(\mathbb{C}_+),
$$
where $P_+$ denotes the orthogonal projection of $L^2(\mathbb{R})$ onto
$H^2(\mathbb{C}_+)$. Then, for $\varphi\in L^\infty(\mathbb{R})$,
$\|\varphi\|_\infty\leq 1$,  the de Branges--Rovnyak space
$\mathcal{H}(\varphi)$ associated to $\varphi$ consists of those
functions in $H^2(\mathbb{C}_+)$ which are in the range of the operator
$(Id-T_\varphi T_{\overline\varphi})^{1/2}$. It is a Hilbert space
when equipped with the inner product
\[
\langle\, (Id-T_\varphi T_{\overline \varphi})^{1/2}f, \, (Id-T_\varphi
T_{\overline \varphi})^{1/2}g \,\rangle_\varphi=\langle f,g
\rangle_2,
\]
where $f,g\in H^2(\mathbb{C}_+)\ominus \hbox{ker }(Id-T_\varphi
T_{\overline \varphi})^{1/2}$.
In what follows we always assume that $\varphi=b$ is an analytic
function in the unit ball of $H^\infty(\mathbb{C}_+)$. In this case, if
\begin{eqnarray}
\label{eq:noyau-reproduisant}
k^b_\omega(z) := \frac{1-\overline{b(\omega)}
b(z)}{z-\overline \omega}, \hspace{1cm}
\omega \in \mathbb{C}_+,
\end{eqnarray}
then we have $\langle f, k^{b}_\omega\rangle_b = 2\pi i f(\omega)$
for all $f\in \mathcal{H}(b)$. In other words, $\mathcal{H}(b)$ is
a reproducing kernel Hilbert space.

These spaces (and, more precisely, their general vector-valued
version) were introduced by de Branges and Rovnyak
\cite{de-branges1, de-branges2} as universal model spaces for
Hilbert space contractions. Thanks to the
pioneer works of Sarason, we know that de Branges--Rovnyak spaces play an
important role in numerous questions of complex analysis and
operator theory (e.g. see \cite{AR,HSS,sarason,Shapiro1,Shapiro2}).
For the general theory of $\mathcal{H}(b)$ spaces we refer to
\cite{sarason}.

In the special case where $b=\Theta$ is an inner
function (that is, $|\Theta|=1$ a.e. on $\mathbb{R}$), the operator
$(Id-T_\Theta T_{\overline\Theta})^{1/2}$ is an orthogonal projection and
$\mathcal{H}(\Theta)$ becomes a closed (ordinary) subspace
of $H^2(\mathbb{C}_+)$ which coincides with the so-called model subspace
\[
K_\Theta^2=H^2(\mathbb{C}_+)\ominus \Theta
H^2(\mathbb{C}_+)=H^2(\mathbb{C}_+)\cap \Theta \, \overline{H^2(\mathbb{C}_+)}
\]
(for the model space theory see \cite{Nikolski}).
We mention one important particular
class of model spaces. If $\Theta(z)=\exp(iaz)$, $a>0$, then
$\mathcal{H}(\Theta)=K_\Theta^2=H^2(\mathbb{C}_+)\cap PW_a^2$,
where $PW_a^2$ stands for the Paley--Wiener space of all entire functions of
exponential type at most $a$, whose restrictions to $\mathbb{R}$ belong to
$L^2(\mathbb{R})$. Then the famous Bernstein's inequality asserts that
\[
\|f'\|_2\leq a \|f\|_2,\qquad f\in PW_a^2.
\]
This classical and important inequality was extended by many authors in many
different directions. It is impossible to give an exhaustive list of
references, but we would like to mention
\cite{Boas,Gorin80,Pe99,Ra90,Ra07,To91}
and \cite[Lecture 28]{Levin}.

Notably, one natural direction is to extend Bernstein's inequality to
general model subspaces. In \cite{Levin1}, Levin showed that if
$\Theta$ is an inner function and $|\Theta'(x)|<\infty$, $x\in \mathbb{R}$,
then for each function $f\in K_\Theta^\infty=H^\infty(\mathbb{C}_+)\cap\Theta
\overline{H^\infty(\mathbb{C}_+)}$, the derivative $f'(x)$ exists in the sense of
nontangential boundary values and
\[
|f'(x)/\Theta'(x)| \leq \|f\|_\infty.
\]
Differentiation in the model spaces $K_\Theta^p:=H^p(\mathbb{C}_+)\cap \Theta
\overline{H^p(\mathbb{C}_+)}$, $1<p<\infty$, was studied extensively by
Dyakonov \cite{Dyak91, Dyak02}, who showed that the Bernstein-type
inequality $\|f'\|_p\le C \|f\|_p$, $f\in K_\Theta^p$, holds
if and only if $\Theta'\in L^\infty(\mathbb{R})$.
Recently, Baranov \cite{Baranov03, Baranov05, Baranov06} has obtained
weighted Bernstein-type inequalities for the model
subspaces $K_\Theta^p$, which generalized previous results of Levin and
Dyakonov. More precisely, for a general inner function $\Theta$,
he proved estimates of the form
\begin{equation}\label{eq:intro}
\|f^{(n)}w_{p,n}\|_{L^p(\mu)} \le C \, \|f\|_p, \qquad f\in K_\Theta^p,
\end{equation}
where $n\geq 1$, $\mu$ is a Carleson measure in the closed upper half-plane
and $w_{p,n}$ is some weight related to the norm of reproducing kernels
of the space $K_\Theta^2$ which
compensates possible growth of the derivative near the boundary.

One of the main ingredients in the results of Dyakonov and
Baranov was an integral formula for the derivatives of functions in
$K_\Theta^p$. Using Cauchy formula, it is easy to see that if $\Theta$ is inner,
$\omega\in\mathbb{C}_+$, $n$ is a non-negative integer and $f\in K_\Theta^p$, then
we have
\begin{eqnarray}\label{eq:fimodel}
f^{(n)}(\omega)=\frac{1}{2\pi i}
\int_\mathbb{R} f(t) \, \overline{k_{\omega,n}^\Theta(t)}\,dt,
\end{eqnarray}
where
\begin{eqnarray}\label{eq:noyaumodel}
\frac{k_{\omega,n}^\Theta(z)}{n!}:= \displaystyle\frac
{1-\Theta(z)\displaystyle\sum_{p=0}^n\frac{\overline{\Theta^{(p)}(\omega)}}{p!}(z-\overline\omega)^p}{(z-\overline{\omega})^{n+1}}\,,\qquad
z\in\mathbb{C}_+.
\end{eqnarray}
A natural question is whether one can extend the formula
(\ref{eq:fimodel}) to boundary points $x_0$.
If $x_0 \in \mathbb{R}$
does not belong to the boundary spectrum $\sigma(\Theta)$ of $\Theta$
(see the definition in Section 5),
then $\Theta$ and all
functions of $K_\Theta^p$ are analytic through a neighborhood of $x_0$ and
then it is obvious that (\ref{eq:fimodel}) is valid
for $z=x_0$. More generally, if $x_0$ satisfies
\begin{eqnarray} \label{eq:interieur-ahern-clark}
\sum_k\frac{\Ima z_k}{|x_0-z_k|^{(n+1)q}}+\int_\mathbb{R} \,
\frac{d\mu(t)}{|t-x_0|^{(n+1)q}}<+\infty,
\end{eqnarray}
then, by the results of Ahern and Clark \cite{ak71} (for $p=2$)
and Cohn \cite{Cohn86} (for $p>1$),
the formula (\ref{eq:fimodel}) is still valid at the
point $x_0\in\mathbb{R}$ for any $f\in K_\Theta^p$
(here $\{z_k\}$ is the sequence of zeros of $\Theta$ and
$\mu$ is the singular measure associated to $\Theta$).
Recently Fricain and Mashreghi studied the boundary behavior
of functions in de Branges--Rovnyak spaces $\mathcal{H}(b)$
and obtained a generalization of representation
(\ref{eq:fimodel})  \cite{Fricain-Mashreghi,Fricain-Mashreghi2}.

In the present paper de Branges--Rovnyak spaces are studied from
the point of view of function theory. Namely, we are interested
in boundary properties of the elements of $\mathcal{H}(b)$ and of their
derivatives, and we establish a number of weighted
Bernstein-type inequalities. Our first goal is to exploit the generalization of representation
(\ref{eq:fimodel}) and obtain an analogue of Bernstein-type
inequality (\ref{eq:intro}) for the de Branges--Rovnyak spaces $\mathcal{H}(b)$, where
$b$ is an {\it arbitrary function} in the unit ball of $H^\infty(\mathbb{C}_+)$ (not
necessarily inner). It should be noted that
the inner product in $\mathcal{H}(b)$ is not given by a usual integral formula. This fact causes certain difficulties. For example, we will see that one has
to add one more term to formula (\ref{eq:fimodel}) in the general case.
In what follows we try to emphasize the points where there is a
difference with the inner case, and suggest a few open questions.

Our second goal is to provide several applications of these
Bernstein-type inequalities. The classical Carleson embedding theorem gives a
simple geometrical condition on a measure $\mu$ in the closed upper half-plane such
that the embedding $H^p(\mathbb{C}_+) \subset L^p(\mu)$ holds. A similar question for
model subspaces $K_\Theta^p$ was studied by Cohn \cite{Cohn82} and
then by Volberg and Treil \cite{VolbergTreil}.
An approach based on the (weighted norm) Bernstein
inequalities for model subspaces $K_\Theta^p$ was suggested in
\cite{Baranov05}. Given $b$
in the unit ball of $H^\infty(\mathbb{C}_+)$, we describe a class of Borel
measures $\mu$ in $\mathbb{C}_+ \cup\mathbb{R}$ such that
$\mathcal{H}(b)\subset L^2(\mu)$. We obtain a geometric condition on $\mu$
sufficient for such embedding. This result generalizes the previous results of Cohn and Volberg--Treil.

Another application concerns the problem of
stability of Riesz bases consisting of reproducing kernels
of $\mathcal{H}(b)$. This problem is connected with the famous problem of bases
of exponentials in $L^2$ on an interval which goes back to Paley and Wiener
\cite{PW34}.
Exponential bases were described by Pavlov \cite{P79}
and by Hruschev, Nikolski and Pavlov in \cite{HNP81}, where
functional model methods have been used. This approach
has been proved fruitful; it has allowed both to recapture all the
classical results and to extend them to general model spaces
(for a detailed presentation of the subject see \cite{Nikolski}).
Fricain has pursued this investigation with respect to bases of reproducing kernels in
vector-valued model spaces \cite{Fricain-JOT} and in de Branges--Rovnyak
spaces \cite{Fricain} where some criteria for a
family of reproducing kernels to be a Riesz basis were obtained.
However, the criteria mentioned
above involve some properties of a given family of
reproducing kernel that are rather difficult to verify. On the other hand, in
many cases, the given family is a slight perturbation of another family of
reproducing kernels that is known to be a basis. This gives rise to the following
stability problem: {\it Given a Riesz basis of reproducing kernels
$(k_{\lambda_n}^b)_{n\geq 1}$ of $\mathcal{H}(b)$, characterize
perturbations of frequencies $(\lambda_n)_{n\geq 1}$ which preserve
the property to be a Riesz basis.}

This problem was also studied by
many authors in the context of exponential bases (see e.g. \cite{Kadec,Red})
and of model subspaces $K_\Theta^2$ \cite{Baranov05a,Cohn-JOT86,Fricain-JOT}.
In the present paper, using the weighted norm inequalities (\ref{eq:intro})
we extend the results about stability in pseudohyperbolic metrics from
\cite{Baranov05a,Fricain-JOT} to de Branges-Rovnyak spaces.

The paper is organized as follows. Sections 2 and 3 contain some
preliminaries concerning integral representations for the $n$-th derivative
of functions in de Branges--Rovnyak spaces. In Section 4 we prove
our first main result, a Bernstein-type inequality for $\mathcal{H}(b)$.
Section 5 contains some estimates relating the
weight $w_{p,n}$ involved in Bernstein inequalities to the distances to the
level sets of $|b|$. Section 6 is devoted to embedding theorems. Finally, in
Section 7 we apply the Bernstein inequality to the problem of
stability of Riesz basis of reproducing kernels in $\mathcal{H}(b)$.

In what follows, the letter $C$ will denote
a positive constant and we assume that its value may change.
We write $f\asymp g$ if $C_1g \le f \le C_2g$ for some positive
constants $C_1, C_2$. The set of integers $1,2,\cdots$ will be denoted by $\mathbb{N}$.

\section{Preliminaries}
Let $b$ be in the unit ball of $H^\infty(\mathbb{C}_+)$ and let $b=B I_\mu
O_b$ be its canonical factorization, where
$$
B(z) = \prod_r e^{i\alpha_r}\frac{z-z_r}{z-\overline{z_r}}
$$
is a Blaschke product, the singular inner function $I_\mu$ is given
by
$$
I_\mu(z) = \exp\left(iaz - \frac{i}{\pi}\int_\mathbb{R}
\bigg(\frac{1}{z-t} + \frac {t}{t^2+1}\bigg)\,d\mu(t)\right)
$$
with a positive singular measure $\mu$ and $a\ge 0$,
and $O_b$ is the outer function
$$
O_b(z) = \exp\left(\frac{i}{\pi}\int_\mathbb{R} \bigg(\frac{1}{z-t} + \frac
{t}{t^2+1}\bigg) \log|b(t)|\,dt\right).
$$
Then the modulus of the angular derivative of $b$ at a point $x \in
\mathbb{R}$ is given by
\begin{equation}\label{eq:Rderivee}
|b'(x)|=a+\sum_r\frac {2\Ima z_r}{|x-z_r|^{2}}+ \frac{1}{\pi}\int_\mathbb{R}
\frac {d\mu(t)}{|x-t|^{2}}+ \frac{1}{\pi}\int_\mathbb{R} \frac
{\big|\log|b(t)|\big|}{|x-t|^{2}}\,dt.
\end{equation}
Hence, we are motivated to define
\begin{eqnarray}\label{eq:condition-type-ahern-clark}
S_n(x):=\sum_{r=1}^{+\infty}\frac {\Ima z_r}{|x-z_r|^{n}}+ \int_\mathbb{R}
\frac {d\mu(t)}{|x-t|^{n}}+ \int_\mathbb{R} \frac
{\big|\log|b(t)|\big|}{|x-t|^{n}}\,dt,
\end{eqnarray}
and
$$
E_n(b):=\{x\in\mathbb{R}:S_n(x)<+\infty\}.
$$
The formula (\ref{eq:Rderivee}) explains why the quantity $S_2$ is
of special interest.

We will need the following simple estimate.

\begin{Lem}\label{Lem:derivee}
For any $x\in \mathbb{R}$, $y>0$, we have $|b'(x+iy)| \le |b'(x)|$.
\end{Lem}

\begin{proof}
Let $z=x+iy$, $y>0$, and assume that $b$ is outer,
\[
b(z)=\exp\left(\frac{i}{\pi}\int_\mathbb{R} \bigg(\frac{1}{z-t} + \frac
{t}{t^2+1}\bigg) \log|b(t)|\,dt\right).
\]
Then
\[
b'(z)=-b(z)\frac{i}{\pi}\int_\mathbb{R}\frac{\log|b(t)|}{(t-z)^2}\,dt,
\]
and clearly
\[
|b'(z)|\leq
\frac{1}{\pi}\int_\mathbb{R}\frac{|\log|b(t)||}{|t-z|^2}\,dt\leq\frac{1}{\pi}\int_\mathbb{R}\frac{|\log|b(t)||}{|t-x|^2}\,dt=|b'(x)|,
\]
by (\ref{eq:Rderivee}). The estimates for inner factors are analogous and
left to the reader (recall that $|b'(x)|=|O_b'(x)|+|I_\mu'(x)|+|B'(x)|$,
$x\in\mathbb{R}$).

\end{proof}

Ahern and Clark \cite{ak71-bis} showed that if $x_0\in E_{n}(b)$, then
$b$ and all its derivatives up to order $n-1$ have (finite)
nontangential limits at $x_0$. In \cite{Fricain-Mashreghi}, we showed that
if $x_0\in E_{2n+2}(b)$ where $n\in \mathbb{Z}_+=\mathbb{N}\cup\{0\}$, then,
for each $f\in\mathcal{H}(b)$ and for each $0 \leq j \leq n$, the nontangential limit
\[
f^{(j)}(x_0):=\lim_{\substack{z \longrightarrow x_0\\\sphericalangle\,\,}}
f^{(j)}(z)
\]
exists. This is a generalization of the Ahern--Clark theorem \cite{ak71} for the
elements of model subspaces $K_\Theta^2$, i.e. for the case when
$b=\Theta$ is an inner function. Moreover, for every $z_0\in
\mathbb{C}_+\cup E_{2n+2}(b)$ and for every function $f\in\mathcal{H}(b)$, we
obtained in \cite{Fricain-Mashreghi2} the following integral representation
for $f^{(n)}(z_0)$.
Let $\rho(t) =  1-|b(t)|^2$ and let
$H^2(\rho)$ be the span  of the Cauchy kernels
$k_z$, $z\in \mathbb{C}_+$, in $L^2(\rho)$
(recall that $k_z(\omega) = (\omega-\overline z)^{-1}$).
Consider the operator
\[
\begin{array}{ccc}
\widetilde{T}_\rho: L^2(\rho) & \longrightarrow &  H^2(\mathbb{C}_+)\\
q & \longmapsto & P_+(q\rho).
\end{array}
\]
We know from \cite[II-3, III-2]{sarason} that if
$f\in \mathcal{H}(b)$ then there exists a (unique) function $g$ in $H^2(\rho)$
such that $T_{\overline b}f=\widetilde{T}_\rho g$. It was shown
in \cite{Fricain-Mashreghi2} that, for $z_0\in
\mathbb{C}_+\cup E_{2n+2}(b)$, $n\in \mathbb{Z}_+$, we have
\begin{eqnarray}\label{eq:representation-de-Branges}
f^{(n)}(z_0)=\frac{n!}{2\pi i }\left(\int_\mathbb{R}
f(t)\overline{k_{z_0,n}^b(t)}\,dt +\int_\mathbb{R}
g(t)\rho(t)\overline{k_{z_0,n}^\rho(t)}\,dt\right),
\end{eqnarray}
where  $k_{z_0,n}^b$
is the function in $\mathcal{H}(b)$ defined by
\begin{equation}\label{eq:kernel1-demiplan}
k_{z_0,n}^b(z):=\frac{1-b(z)\displaystyle\sum_{j=0}^n\frac{\overline{b^{(j)}(z_0)}}{j!}(z-\overline{z_0})^j}{(z-\overline{z_0})^{n+1}},\qquad
z\in\mathbb{C}_+,
\end{equation}
and $k_{z_0,n}^\rho$ is the function in $L^2(\rho)$ defined by
\begin{equation}\label{eq:kernel2-demiplan}
k_{z_0,n}^\rho(t):=\frac{\displaystyle\sum_{j=0}^n\frac{\overline{b^{(j)}(z_0)}}{j!}(t-\overline{z_0})^j}{(t-\overline{z_0})^{n+1}},\qquad
t\in\mathbb{R}.
\end{equation}

Note that if $b$ is inner, then $\rho\equiv 0$ and thus
(\ref{eq:representation-de-Branges}) reduces to (\ref{eq:fimodel})
which was the key representation formula used in
\cite{Baranov03,Baranov05,Dyak91,Dyak02} to
obtain Bernstein-type inequalities for model subspaces $K_\Theta^p$. If $n=0$ then $k_{z_0,0}^b$ corresponds  to the reproducing kernel of $\mathcal{H}(b)$ defined in (\ref{eq:noyau-reproduisant}).

\section{A new representation formula for the derivatives}

We start with a slight modification of the representation
(\ref{eq:representation-de-Branges}) for $n\in \mathbb{N}$.

\begin{Prop}\label{prop:nouvelle-representation}
Let $b$ be in the unit ball of $H^\infty(\mathbb{C}_+)$. Let
$z_0\in\mathbb{C}_+\cup E_{2n+2}(b)$, $n\in \mathbb{N}$, and let
\begin{equation}\label{eq:nouveau-noyau}
\mathfrak{K}_{z_0,n}^\rho(t):=\overline{b(z_0)} \, \,
\dfrac{\sum_{j=0}^n \, \binom{n+1}{j+1} \, (-1)^j \,
{\overline{b^j(z_0)}} \, b^j(t)}{(t-\overline{z_0})^{n+1}},\qquad
t\in\mathbb{R}.
\end{equation}
Then ${(k_{z_0}^b)}^{n+1}\in H^2(\mathbb{C}_+)$ and $\mathfrak{K}_{z_0,n}^\rho
\in L^2(\rho)$. Moreover, for every function $f\in\mathcal{H}(b)$, we have
\begin{eqnarray}\label{eq:integrale-cle}
f^{(n)}(z_0)=\frac{n!}{2\pi i }\left(\int_\mathbb{R}
f(t)\overline{(k_{z_0}^b)^{n+1}(t)}\,dt +\int_\mathbb{R}
g(t)\rho(t)\overline{\mathfrak{K}_{z_0,n}^\rho(t)}\,dt\right),
\end{eqnarray}
where $g\in H^2(\rho)$ is such that $T_{\overline
b}f=\widetilde{T}_\rho g$.
\end{Prop}

\begin{proof}
Let $a_j=b^{(j)}(z_0)/j!$. Then
\begin{align}\label{eq:recurrence-relation}
k_{z_0,\ell }^b(z) =&\,\, \dfrac{1-\overline{b(z_0)}b(z)-b(z)\displaystyle\sum_{j=1}^\ell
\overline{a_j}(z-\overline{z_0})^j}{(z-\overline{z_0})^{\ell +1}}\nonumber\\
=&\,\,\frac{1-\overline{b(z_0)}b(z)}{(z-\overline{z_0})^{\ell
+1}}-b(z)\displaystyle\sum_{j=1}^\ell
\frac{\overline{a_j}}{(z-\overline{z_0})^{\ell +1-j}}\nonumber.
\end{align}
Hence, multiplying by $(1-\overline{b(z_0)}b(z))^\ell$, we obtain
\begin{equation}\label{eq:recurrence-relation}
(k_{z_0}^b)^{\ell +1}(z)= (1-\overline{b(z_0)}b(z))^\ell \,
k_{z_0,\ell }^b(z) +b(z)\sum_{j=1}^\ell
\overline{a_j}(1-\overline{b(z_0)}b(z))^{j-1}(k_{z_0}^b)^{\ell
+1-j}(z).
\end{equation}
Since $z_0\in\mathbb{C}_+\cup E_{2n+2}(b)$, according to \cite[Proposition
3.1 and Lemma 3.2]{Fricain-Mashreghi2}, the functions $k_{z_0}^b$
and $k_{z_0,\ell}^b$ ($1\leq \ell \leq n$) belong to $\mathcal{H}(b)$.
Hence, using the recurrence relation (\ref{eq:recurrence-relation})
and that $1-\overline{b(z_0)}b(z)\in H^\infty(\mathbb{C}_+)$, we see
immediately by induction that $(k_{z_0}^b)^{n+1}\in H^2(\mathbb{C}_+)$.

We prove now that $\mathfrak{K}_{z_0,n}^\rho\in L^2(\rho)$. Write
$\mathfrak{K}_{z_0,n}^\rho(t)=(t-\overline{z_0})^{-(n+1)}\varphi(t)$,
with
\[
\varphi(t)=\overline{b(z_0)}\sum_{j=0}^n\binom{n+1}{j+1}(-1)^j{\overline{b^j(z_0)}}b^j(t).
\]
Since $\varphi\in L^\infty(\mathbb{R})$, it is sufficient to prove that
$(t-\overline{z_0})^{-(n+1)}\in L^2(\rho)$. If $z_0\in\mathbb{C}_+$, this
fact is trivial and if $z_0\in E_{2n+2}(b)$, the inequality
$1-x\less |\log x|$, $x\in(0,1]$, implies
\[
\int_\mathbb{R}\frac{1-|b(t)|^2}{|t-z_0|^{2n+2}}\,dt\less
2\int_\mathbb{R}
\frac{\big|\log|b(t)|\big|}{|t-z_0|^{2n+2}}\,dt<+\infty
\]
which is the required result.

It remains to prove (\ref{eq:integrale-cle}). Let $\psi$ be any
element of $H^2(\mathbb{C}_+)$. According to
(\ref{eq:representation-de-Branges}), we have
\begin{align*}
\frac{2\pi i}{n!} \, f^{(n)}(z_0) =& \langle f, k_{z_0,n}^b
\rangle_2+\langle \rho g,k_{z_0,n}^\rho\rangle_2\\
=& \langle f, k_{z_0,n}^b-b\psi \rangle_2+\langle \bar{b}f, \psi
\rangle_2+\langle \rho g,k_{z_0,n}^\rho\rangle_2.
\end{align*}
But we have $T_{\overline b}f=\widetilde{T_\rho}g$, which means that
$\overline{b}f-\rho g\perp H^2(\mathbb{C}_+)$. Since $\psi\in H^2(\mathbb{C}_+)$,
it follows that $\langle \overline{b}f,\psi \rangle_2=\langle \rho
g,\psi \rangle_2$. Hence the identity
\begin{equation} \label{E:bpsi}
\frac{2\pi i }{n!}f^{(n)}(z_0)=\langle f, k_{z_0,n}^b-b\psi
\rangle_2+\langle \rho g,k_{z_0,n}^\rho+\psi \rangle_2
\end{equation}
holds for each $\psi \in H^2(\mathbb{C}_+)$. A very specific $\psi$ gives us  the
required representation. To find $\psi$ note that, on one hand, we have
\begin{eqnarray*}
k_{z_0,n}^b(t) - (k_{z_0}^b)^{n+1}(t) &=& \frac{1-b(t)\sum_{j=0}^n\overline{a_j}(t-\overline{z_0})^j-(1-\overline{b(z_0)}
b(t))^{n+1}}{(t-\overline{z_0})^{n+1}}\\
&=& \frac{1-(1-\overline{b(z_0)}
b(t))^{n+1}}{(t-\overline{z_0})^{n+1}} -b(t)
\frac{\sum_{j=0}^n\overline{a_j}
(t-\overline{z_0})^j}{(t-\overline{z_0})^{n+1}} = b(t)\psi(t),
\end{eqnarray*}
where
$$
\psi(t) = \frac{\sum_{j=1}^{n+1} (-1)^{j+1} \binom{n+1}{j}
(\overline{b(z_0)})^{j} (b(t))^{j-1}} {(t-\overline{z_0})^{n+1}} -
\frac{\sum_{j=0}^n\overline{a_j}
(t-\overline{z_0})^j}{(t-\overline{z_0})^{n+1}}.
$$
On the other hand, we easily see that
\begin{eqnarray*}
k_{z_0,n}^\rho(t)+\psi(t) &=& \frac{\sum_{j=1}^{n+1} (-1)^{j+1}
\binom{n+1}{j} (\overline{b(z_0)})^{j} (b(t))^{j-1}}
{(t-\overline{z_0})^{n+1}}\\
&=& \overline{b(z_0)} \frac{\sum_{j=0}^n (-1)^{j} \binom{n+1}{j+1}
(\overline{b(z_0)})^{j} (b(t))^{j}} {(t-\overline{z_0})^{n+1}} =
\mathfrak{K}_{z_0,n}^\rho(t).
\end{eqnarray*}
Therefore, (\ref{eq:integrale-cle}) follows immediately from
(\ref{E:bpsi}).

\end{proof}

We now introduce the weight involved in our Bernstein-type
inequalities. Let $1< p\leq 2$ and let $q$ be its conjugate
exponent. Let $n\in \mathbb{N}$.
Then, for $z\in\overline{\mathbb{C}_+}$, we define
$$
w_{p,n}(z):=\min\left\{\, \|{(k_z^b)}^{n+1}\|_q^{-pn/(pn+1)}, \,\,
\|\rho^{1/q}\mathfrak{K}_{z,n}^\rho\|_q^{-pn/(pn+1)} \,\right\};
$$
we assume $w_{p,n}(x)=0$, whenever $x\in \mathbb{R}$ and at least
one of the functions $(k_x^b)^{n+1}$ or $\rho^{1/q}
\mathfrak{K}_{x,n}^\rho$ is not in $L^q(\mathbb{R})$. In what follows we will
write $w_p$ for $w_{p,1}$.

The choice of the weight is motivated by representation
(\ref{eq:integrale-cle}) which shows that the quantity
$\max \big\{ \|{(k_z^b)}^{n+1}\|_2,
\, \|\rho^{1/2}\mathfrak{K}_{z,n}^\rho\|_2 \big\}$ is related to
the norm of the functional $f\mapsto f'(z)$ on $\mathcal{H}(b)$.
Moreover, we strongly believe that the norms of reproducing
kernels are an important characteristic of the space $\mathcal{H}(b)$
which captures many geometric properties of $b$
(see Section 5 for certain estimates confirming this point).

Using similar arguments as in the proof of Proposition
\ref{prop:nouvelle-representation}, it is easy to see
that $\rho^{1/q}\mathfrak{K}_{x,n}^\rho \in
L^q(\mathbb{R})$ if $x\in E_{q(n+1)}(b)$. It is also natural to
expect that  $(k_x^b)^{n+1}\in L^q(\mathbb{R})$ for $x\in
E_{q(n+1)}(b)$. This is true when $b$ is an inner function, by a
result of Cohn \cite{Cohn86}, and for a general function $b$ with
$q=2$ by (\ref{eq:recurrence-relation}) and \cite[Lemma
3.2]{Fricain-Mashreghi}. However, it seems that the methods of
\cite{Cohn86} and \cite{Fricain-Mashreghi} do not apply in the general case.

\begin{question}\label{quest-0}
{\rm Is it true that for $x\in\mathbb{R}$,
$(k_x^b)^{n+1}\in L^q(\mathbb{R})$ if $x\in E_{q(n+1)}(b)$? }
\end{question}

\begin{remark}
{\rm If $f\in\mathcal{H}(b)$ and $1< p\leq 2$, then $(f^{(n)}w_{p,n})(x)$ is
well-defined on $\mathbb{R}$. It follows from the \cite{Fricain-Mashreghi}
that $f^{(n)}(x)$ and $w_{p,n}(x)$ are finite if $S_{2n+2}(x)<+\infty$. If
$S_{2n+2}(x)=+\infty$, then $\|(k_{x}^b)^{n+1}\|_2=+\infty$. Hence,
$\|(k_{x}^b)^{n+1}\|_q=+\infty$ which, by definition, implies $w_{p,n}(x)=0$,
and thus we may assume $(f^{(n)}w_{p,n})(x)=0$.}
\end{remark}

\begin{remark}
{\rm In the inner case, we have $\rho(t)\equiv 0$ and the second term in the
definition of the weight $w_{p,n}$ disappears. It should be emphasized
that in the general case both terms are essential: below we show
(Example \ref{ex:secondterm}) that the
norm $\|\rho^{1/q}\mathfrak{K}_{z,n}^\rho\|_q$
can not be majorized uniformly by the norm $\|{(k_z^b)}^{n+1}\|_q$.}
\end{remark}

\begin{Lem}\label{lem:minoration}
For $1< p\leq 2$, $n\in \mathbb{N}$, there is a constant $A=A(p,n)>0$ such that
$$
w_{p,n}(z)\geq A\,\frac{(\Ima z)^n}{(1-|b(z)|)^{\frac{pn}{q(pn+1)}}}, \qquad
z\in\mathbb{C}_+.
$$
\end{Lem}

\begin{proof}
On one hand, note that
\begin{align*}
\|{(k_z^b)}^{n+1}\|_q^q = \,\, &
\int_\mathbb{R} \left|\frac{1-\overline{b(z)}b(t)}{t-\overline z}
\right|^{(n+1)q}\,dt \leq \,\, \frac{C}{(\Ima z)^{(n+1)q-2}} \int_\mathbb{R}\left|
\frac{1-\overline{b(z)}b(t)}{t-\overline z}\right|^2\,dt\\
 = \,\,& \frac{C}{(\Ima z)^{(n+1)q-2}}\|k_z^b\|_b^2\leq
C\,\,\frac{1-|b(z)|}{(\Ima z)^{(n+1)q-1}}.
\end{align*}

On the other hand, we have
\begin{align*}
\|\rho^{1/q}\mathfrak{K}_{z,n}^\rho\|_q^q
= \,\, & \int_\mathbb{R}
\bigg| \frac{b(z)\sum_{j=0}^n\binom{n+1}{j+1}(-1)^j\overline{b(z)}^jb^j(t)}{(t-\overline
z)^{n+1}}\bigg|^q\,(1-|b(t)|^2)\,dt \\
\leq \,\, & \frac{C}{(\Ima z)^{(n+1)q-2}}\,\,\int_\mathbb{R}
\frac{1-|b(t)|}{|t-z|^2}\,dt.
\end{align*}
If $|b(z)|<1/2$, then we obviously have
\[
\int_\mathbb{R}\frac{1-|b(t)|}{|t-z|^{2}}dt\le C\frac{1-|b(z)|}{\Ima z},
\]
and if $|b(z)|\geq 1/2$, using  $1-|b(t)| \le \big|\log |b(t)|\big|$, we get
\[
\Ima z \int_\mathbb{R}\frac{1-|b(t)|}{|t-z|^{2}}dt
\le
\Ima z \int_\mathbb{R}\frac{\big|\log |b(t)|\big|}{|t-z|^{2}}dt=
\pi \log \frac{1}{|O_b(z)|} \asymp 1-|O_b(z)|,
\]
since $|O_b(z)|\ge |b(z)|\ge 1/2$. We recall that $O_b$ is the outer
part of $b$. Therefore, in any case we have
\[
\int_\mathbb{R}\frac{1-|b(t)|}{|t-z|^{2}}dt\le C\frac{1-|b(z)|}{\Ima z},
\]
and we get
\[
\|\rho^{1/q}\mathfrak{K}_{z,n}^\rho\|_q^q\leq C\,\,\frac{1-|b(z)|}{(\Ima
z)^{(n+1)q-1}}. \]
To complete the proof, it suffices to note that $\frac{(n+1)q-1}{q}=
n+\frac 1p=\frac{np+1}{p}$.

\end{proof}

Representation formulae discussed above reduce the study of
differentiation in de Branges--Rovnyak spaces $\mathcal{H}(b)$ to the study
of certain integral operators.

\section{Bernstein-type inequalities}

A Borel measure $\mu$ in the closed upper half-plane
$\overline{\mathbb{C}_+}$ is said to be a Carleson measure if there is a
constant $C_\mu>0$ such that
\begin{equation}
\label{carl} \mu(\, S(x,h) \,)\le C_\mu \, h,
\end{equation}
for all squares $S(x,h)=[x,x+h]\times[0,h]$, $x\in\mathbb{R}$, $h>0$,
with the lower side on the real axis. We denote the class of Carleson
measures by $\mathcal{C}$. Recall that, according to a classical
theorem of Carleson, $\mu \in \mathcal{C}$ if and only if
$H^p(\mathbb{C}_+) \subset L^p(\mu)$ for some (all) $p>0$.

One of our main results in this paper is the following
Bernstein-type inequality.

\begin{Thm}
\label{maintheorem}
Let $\mu \in \mathcal{C}$, let $n\in\mathbb{N}$, let $1< p\leq 2$, and let
$$(T_{p,n} f)(z)=f^{(n)}(z)w_{p,n}(z),\qquad f\in\mathcal{H}(b).$$
If $1<p<2$, then $T_{p,n}$
is a bounded operator from $\mathcal{H}(b)$ to $L^2(\mu)$,
that is, there is a constant
$C=C(\mu,p,n)>0$ such that
\begin{equation}
\label{eq:mainineq}
\|f^{(n)}w_{p,n}\|_{L^2(\mu)} \le C\|f\|_b,
\qquad f\in \mathcal{H}(b).
\end{equation}
If $p=2$, then $T_{2,n}$ is of weak type $(2,2)$ as an
operator from $\mathcal{H}(b)$ to $L^2(\mu)$.
\end{Thm}

\begin{proof}
According to Proposition \ref{prop:nouvelle-representation},
for all $z\in\overline{\mathbb{C}_+}$ and any
function $f\in\mathcal{H}(b)$, we have
\begin{equation}\label{eq:depart}
\frac{2\pi i }{n!}f^{(n)}(z)w_{p,n}(z)=w_{p,n}(z)\int_\mathbb{R} f(t)
\overline{({k_{z_0}^b})^{n+1}(t)}\,dt +w_{p,n}(z) \int_{\mathbb{R}}
g(t)\rho(t)\overline{\mathfrak{K}_{z,n}^\rho(t)}\,dt.
\end{equation}
Let
$$
w_{p,n}^{(1)}(z):=\|{(k_z^b)}^{n+1}\|_q^{-pn/(pn+1)}, \ \qquad
w_{p,n}^{(2)}(z):=\|\rho^{1/q}\mathfrak{K}_{z,n}^\rho\|_q^{-pn/(pn+1)},
$$
where we assume that $w_{p,n}^{(i)}(z)=0$ if the corresponding
integrand is not in $L^q(\mathbb{R})$, and put
$h_i(z)=(w_{p,n}^{(i)}(z))^{1/n}$, $i=1,2$. We remind that
\[
w_{p,n}(z) = \min \{\, w_{p,n}^{(1)}(z), \, w_{p,n}^{(2)}(z) \}.
\]
We split each of the two integrals in (\ref{eq:depart}) into two
parts, i.e.
$$\frac{2\pi i }{n!}f^{(n)}(z)w_{p,n}(z)=I_1f(z)+I_2f(z)+I_3g(z)+I_4g(z),$$
where
$$I_1f(z)= w_{p,n}(z)\int_{|t-z|\geq h_1(z)}
f(t)\overline{({k_{z}^b})^{n+1}(t)}\,dt,$$
$$I_2f(z)= w_{p,n}(z)\int_{|t-z|< h_1(z)}
f(t)\overline{({k_{z}^b})^{n+1}(t)}\,dt,$$
$$I_3g(z)= w_{p,n}(z)\int_{|t-z|\geq h_2(z)} g(t)\rho(t)
\overline{\mathfrak{K}_{z,n}^\rho(t)}\,dt,$$
$$I_4g(z)= w_{p,n}(z)\int_{|t-z|< h_2(z)} g(t)\rho(t)
\overline{\mathfrak{K}_{z,n}^\rho(t)}\,dt.$$

Note that by Lemma \ref{lem:minoration}, $h_i(z)\geq A \, \Ima z$,
$z\in\mathbb{C}_+$, $i=1,2$. Hence,
\begin{align*}
|I_1f(z)|\leq  &C h_1^n(z)\int_{|t-z|\geq h_1(z)}\frac
{|f(t)|}{|t-z|^{n+1}}\,dt\\
\leq & C h_1(z)\int_{|t-z|\geq h_1(z)}\frac {|f(t)|}{|t-z|^2}\,dt,
\end{align*}
and
\begin{align*}
|I_3g(z)|\leq  & C h_2^n(z)\int_{|t-z|\geq  h_2(z)}\frac
{|g(t)|\rho^{1/2}(t)}{|t-z|^{n+1}}\,dt\\
\leq  & C h_2(z)\int_{|t-z|\geq  h_2(z)}\frac
{|g(t)|\rho^{1/2}(t)}{|t-z|^2}\,dt.
\end{align*}

\noindent Using \cite[Theorem 3.1]{Baranov05}, we see that $I_1 :
L^2(\mathbb{R}) \longrightarrow L^2(\mu)$ and  $I_3:L^2(\rho)
\longrightarrow L^2(\mu)$ are bounded operators. To estimate the
integral $I_2f$, put
\[
K(z,t):=h_1^n(z)|(k_z^b)^{n+1}(t)|.
\]
Then
\begin{align*}
\|K(z,\cdot)\|_q^{-p}=& (h_1(z))^{-pn}\|(k_z^b)^{n+1}\|_q^{-p}\\
=& (h_1(z))^{-pn}(w_{p,n}^{(1)}(z))^{(pn+1)/n}=h_1(z).
\end{align*}
Thus
$$
|I_2f(z)|\leq h_1^n(z)\int_{|t-z|< h_1(z)}
|f(t)||{(k_z^b)}^{n+1}(t)|\,dt=\int_{|t-z|<\|K(z,\cdot)\|_q^{-p}}
|f(t)| K(z,t)\,dt.
$$
Since $\|K(z,\cdot)\|_q^{-p}=h_1(z)\geq A \, \Ima z$, we may apply
\cite[Theorem 3.2]{Baranov05}. Therefore, the operator $I_2$ is of
weak type $(2,2)$ as an operator from $L^2(\mathbb{R})$ to $L^2(\mu)$ if $p=2$
and it is a bounded operator from $L^2(\mathbb{R})$ to $L^2(\mu)$ if $1<p<2$. To
estimate the integral $I_4g$, we use the same technique and put
$$
\kappa(z,t):=\frac {\rho^{1/q}(t)|\mathfrak{K}_{z,n}^\rho(t)|}
{\|\rho^{1/q}\mathfrak{K}_{z,n}^\rho\|_q^{pn/(pn+1)}}.
$$
In other words, $\kappa(z,t)= w_{p,n}^{(2)}(z)\rho^{1/q}(t)|
\mathfrak{K}_{z,n}^\rho(t)|$. Thus
\begin{align*}
|I_4g(z)|\leq
&\,\, w_{p,n}^{(2)}(z)\int_{|t-z|< h_2(z)}
|g(t)|\rho(t)|\mathfrak{K}_{z,n}^\rho(t)|\,dt \\
= & \int_{|t-z|<h_2(z)}|g(t)|\rho^{1/p}(t)\kappa(z,t)\,dt.
\end{align*}
But $\|\kappa(z,\cdot)\|_q^{-p}= (w_{p,n}^{(2)}(z))^{-p}\|\rho^{1/q}
\mathfrak{K}_{z,n}^\rho\|_q^{-p} = h_2(z)$. Hence, we get
\[
|I_4g(z)|\leq
\int_{|t-z|<\|\kappa(z,\cdot)\|_q^{-p}}|g(t)|\rho^{1/p}(t)\kappa(z,t)\,dt.
\]
Since $p\leq 2$ and $\rho(t)\leq 1$, we have
$$
|I_4g(z)|\leq
\int_{|t-z|<\|\kappa(z,\cdot)\|_q^{-p}}|g(t)|\rho^{1/2}(t)\kappa(z,t)\,dt,
$$
and since $\|\kappa(z,\cdot)\|_q^{-p}= h_2(z)\geq A \, \Ima z$, we
may apply again \cite[Theorem 3.2]{Baranov05}. Therefore, the
operator $I_4$ is of weak type $(2,2)$ as an operator from
$L^2(\rho)$ to $L^2(\mu)$ if $p=2$ and it is a bounded operator from $L^2(\rho)$ to $L^2(\mu)$ if $1< p<2$.

To conclude it remains to note that
\[
\|f\|_b^2=\|f\|_2^2+\|g\|_\rho^2,
\]
which implies that the operators $f\mapsto f$ from $\mathcal{H}(b)$ to $H^2(\mathbb{C}_+)$
and $f\mapsto g$ from $\mathcal{H}(b)$ to $L^2(\rho)$ are contractions.

\end{proof}

\begin{example}\label{ex:secondterm}
{\rm We show that for a general function $b$ both terms in the definition
of the weight $w_{p,n}$ are important.
Obviously, for an inner $b$ the norm
$\|\rho^{1/q}\mathfrak{K}_{z,n}^\rho\|_q$ vanishes.
However, for some outer functions $b$ it may be essentially
larger than $\|{(k_z^b)}^{n+1}\|_q$.

Let $\varepsilon\in (0,1)$ and let $b$ be an outer function such
that $|b(t)|=\varepsilon$
for $|t|<1$ and $|b(t)|=1$ for $|t|>1$. Note that $b(z) =
\exp\Big(-\frac{i}{\pi}\log\varepsilon \log\frac{z-1}{z+1}\Big)$,
where $\log$ is the main branch of the logarithm in
$\mathbb{C}\setminus (-\infty, 0]$. We show that
\begin{equation}\label{eq:ex1}
\sup\limits_{y>0} \frac{\|\rho^{1/q}\mathfrak{K}_{iy,1}^\rho\|_q}
{\|{(k_{iy}^b)}^2\|_q}  \longrightarrow \infty \quad \mbox{as}\quad \varepsilon \longrightarrow 1-,
\end{equation}
and so, the second term in the weight $w_{p,1}$
can be dominating. Note that $b(iy)\to\varepsilon$ and $b(t)\to\varepsilon$,
as $y\to 0+$ and $|t|\le \sqrt{y}$.
Hence, for a fixed $\varepsilon$ and sufficiently small $y>0$
we have
\[
\int _{|t|\le\sqrt{y}}  |k_{iy}^b (t)|^{2q} dt
= \int _{|t|\le\sqrt{y}}
\left|\frac{1-\overline{b(iy)}b(t)}{t+iy}\right|^{2q}dt
\le C(1-\varepsilon)^{2q}
\int _{|t|\le\sqrt{y}}
\frac{dt}{|t+iy|^{2q}}.
\]
Thus
\begin{equation}\label{eq:ex2}
\int _{|t|\le\sqrt{y}}
\left|\frac{1-\overline{b(iy)}b(t)}{t+iy}\right|^{2q}dt
\le C\frac{(1-\varepsilon)^{2q}}{y^{2q-1}},
\end{equation}
whereas
\begin{equation}\label{eq:ex3}
\int_{|t|>\sqrt{y}}
\left|\frac{1-\overline{b(iy)}b(t)}{t+iy}\right|^{2q}dt \le Cy^{-q+1/2}.
\end{equation}
On the other hand,
\[
\mathfrak{K}_{iy,1}^\rho(t) =
\overline{b(iy)} \frac{2-\overline{b(iy)} b(t)}{(t+iy)^2},
\]
and so

\[
\|\rho^{1/q}\mathfrak{K}_{iy,1}^\rho\|^q_q \asymp |b(iy)|^q
\int_\mathbb{R} \frac{1-|b(t)|}{|t+iy|^{2q}} \asymp
\frac{1-\varepsilon}{y^{2q-1}}.
\]
Combining the last estimate with (\ref{eq:ex2})
and  (\ref{eq:ex3}), we obtain  (\ref{eq:ex1}). }
\end{example}

\begin{remark}
{\rm It should be emphasized that the constants
in the Bernstein-type inequalities corresponding
to Theorem \ref{maintheorem} depend only on $p,n$
and the Carleson constant $C_\mu$ of the measure $\mu$,
but not on $b$ (the properties of $b$ are contained in the weight
$w_{p,n}$ in the left-hand side of (\ref{eq:mainineq})). }
\end{remark}

\begin{remark}
{\rm All the results stated above have their natural analogues for
the spaces $\mathcal{H}(b)$ in the unit disc.
In particular, Theorem \ref{maintheorem} remains true when we replace
the kernels for the half-plane by the kernels for the disc. The case of inner functions
in the disc is considered in detail in \cite{Baranov06}. }
\end{remark}

\begin{remark}
{\rm An important feature of the de Branges--Rovnyak spaces theory is the difference
between the extreme (i.e. $b$ is an extreme point of the unit ball of $H^\infty(\mathbb{C}_+)$)
and the non-extreme cases. Our Bernstein inequality applies to both cases.
However, in the extreme case one can expect more regularity near the boundary
and this situation is more interesting for us.}
\end{remark}

\section{Distances to the level sets}

To apply Theorem \ref{maintheorem}, one should have effective
estimates for the weight $w_{p,n}$, that is, for the norms
of the reproducing kernels. In this section we relate the weight
$w_{p,n}$ to the distances to the level sets of $|b|$.
We start with some notations. Denote by $\sigma(b)$
the boundary spectrum of $b$, i.e.
$$
\sigma(b) := \Big\{x\in \mathbb{R} :
\liminf_{\substack{z \longrightarrow x\\ z\in\mathbb{C}_+}} |b(z)|<1 \Big\}.
$$
Then, for $b=BI_\mu O_b$,
$\clos \sigma (b)$ is the smallest closed subset of $\mathbb{R}$
containing the limit points of the zeros of the Blaschke product $B$
and the supports of the measures $\mu$ and $\log|b(t)|\, dt$.
It is well known and easy to see that $b$ and any element of $\mathcal{H}(b)$
has an analytic extension through any interval from the
open set $\mathbb{R} \setminus \clos \sigma (b)$.

For $\varepsilon \in (0,1)$, we put
$$
\Omega(b, \varepsilon) :=\{z\in\mathbb{C}_+: |b(z)|< \varepsilon\},
$$
and
$$
\widetilde{\Omega}(b, \varepsilon) :=\sigma(b) \cup \Omega(b, \varepsilon),
$$
where $\sigma(b)$ is the boundary spectrum of $b$.
Finally, for $x\in \mathbb{R}$, we introduce the following three distances
\begin{eqnarray*}
d_0(x) &:=& {\rm dist}\, (x, \sigma(b)), \\
d_\varepsilon(x) &:=& {\rm dist}\, (x, \Omega(b, \varepsilon)), \\
\tilde d_\varepsilon(x) &:=& {\rm dist}\, (x, \widetilde{\Omega}(b,
\varepsilon)).
\end{eqnarray*}

\noindent Note that whenever $b=\Theta$ is an inner function,
for all $x\in \sigma(\Theta)$, we have
\[
\liminf_{\substack{z \longrightarrow x\\ z\in\mathbb{C}_+}} |\Theta(z)|=0,
\]
and thus
$d_\varepsilon(t) = \tilde d_\varepsilon(t)$,
$t\in\mathbb{R}$. However, for an arbitrary function $b$ in the unit ball of
$H^\infty(\mathbb{C}_+)$, we have to distinguish between the distance functions
$d_\varepsilon$ and $\tilde d_\varepsilon$.

\begin{Lem}\label{lem:levelsets1}
There exists a positive constant $C=C(\varepsilon)$
such that, for all $x\in \mathbb{R}\setminus \sigma(b)$,
\[
|b'(x)| \le C \big( \tilde d_\varepsilon(x) \big)^{-1}.
\]
\end{Lem}

\begin{proof}
For the case of an inner function the inequality is proved in
\cite[Theorem 4.9]{Baranov05}. For the general case, let $b=I_bO_b$ be the inner-outer
factorization of $b$. Since $|b'(x)| = |I_b'(x)| +|O_b'(x)|$, $x\in
\mathbb{R}\setminus \sigma(b)$, we may assume, without loss of generality,
that $b$ is outer. Recall that in this case
$$
|b'(x)| = \frac{1}{\pi} \int_\mathbb{R}
\frac{\big|\log|b(t)|\big|}{|t-x|^2} \, dt.
$$

Fix $x\in \mathbb{R}\setminus \sigma(b)$ and suppose $0<y<d_0(x)$. Let
$z=x+iy$. Then
$$
\log\frac{1}{|b(z)|}=\frac{y}{\pi} \int_\mathbb{R}
\frac{\big|\log|b(t)|\big|}{|t-z|^2} \, dt =  \frac{y}{\pi}
\int_{|t-x|\ge d_0(x)} \frac{\big|\log|b(t)|\big|}{|t-z|^2} \, dt.
$$
Since $|t-z|\le |t-x|+ y\le 2|t-x|$ whenever $|t-x|\ge d_0(x)$, we have
$$
\log\frac{1}{|b(z)|} \ge \frac{y}{4\pi} \int_{|t-x|\ge d_0(x)}
\frac{\big|\log|b(t)|\big|}{|t-x|^2} \, dt = \frac{y|b'(x)|}{4}.
$$
Hence
\begin{equation} \label{E:binbp}
|b(x+iy)| \leq \exp \big( - y|b'(x)|/4 \big),
\end{equation}
provided that $0<y<d_0(x)$.

Let $C = 4\log \varepsilon^{-1}$. If $|b'(x)| \le C/|d_0(x)|$, then
the statement is valid since $\tilde d_\varepsilon(x) \leq d_0(x)$. On the other
hand, if $|b'(x)| > C/|d_0(x)|$, then we consider the point $z=
x+iC/|b'(x)|$ for which $\Ima z = C/|b'(x)|< d_0(x)$. Hence, by
(\ref{E:binbp}), we have $|b(z)| \leq \varepsilon$ which immediately
implies $\tilde d_\varepsilon(x) \leq C/|b'(x)|$.

\end{proof}

\begin{Lem}\label{lem:levelsets2}
 For each $p>1$, $n\geq 1$ and $\varepsilon\in (0,1)$, there exists $C=C(\varepsilon,p,n)>0$
such that
\begin{equation}\label{eq:minor1a}
\big(\, \tilde d_\varepsilon (x) \,\big)^n \le C \, w_{p, n} (x+iy),
\end{equation}
for all $x\in\mathbb{R}$ and $y\geq 0$.
\end{Lem}

\begin{proof}
Let $z=x+iy$, $y\geq 0$.
Assume that $x \in \mathbb{R}\setminus \sigma(b)$
(otherwise $\tilde d_\varepsilon(x)=0$ and (\ref{eq:minor1a}) is trivial). Since
$-(n+1)q+1 = -q\frac{np+1}{p}$, the
estimate (\ref{eq:minor1a})
is equivalent to
\begin{eqnarray}\label{eq:minor2}
\int_\mathbb{R} \bigg|\frac{1-\overline{b(z)}b(t)}
{t-\overline z} \bigg|^{(n+1)q}dt \le C (\tilde d_\varepsilon(x))^{-(n+1)q+1},
\end{eqnarray}
and
\begin{eqnarray}\label{eq:minor3}
\int_\mathbb{R} \left|
\frac{\overline{b(z)}\sum_{j=0}^n\binom{n+1}{j+1}(-1)^j\overline{b(z)}^jb^j(t)}{(t-\overline
z)^{n+1}}\right|^q \rho(t)\,dt \le
C (\tilde d_\varepsilon(x))^{-(n+1)q+1}.
\end{eqnarray}
Inequality (\ref{eq:minor3}) is obvious, since
$\rho(t)=0$ if $|t-x|< \tilde d_\varepsilon(x)$. To prove
(\ref{eq:minor2}), we estimate separately the integrals over
$\{t:\,|t-x|\le \tilde d_\varepsilon(x)/2 \}$ and
$\{t:\,|t-x|> \tilde d_\varepsilon(x)/2 \}$. Obviously,
$$
\int_{|t-x|> \tilde d_\varepsilon(x)/2}  \bigg|\frac{1-\overline{b(z)}b(t)}
{t-\overline z} \bigg|^{(n+1)q}dt   \le C (\tilde d_\varepsilon(x))^{-(n+1)q+1}.
$$
Since $|b(t)|=1$ if $|t-x|\le \tilde d_\varepsilon(x)/2$,
for the second integral we have
\begin{align*}
\int_{|t-x|\le \tilde d_\varepsilon(x)/2}
\bigg|\frac{1-\overline{b(z)}b(t)}{t-\overline z} \bigg|^{(n+1)q}dt
= \,\, &
\int_{|t-x|\le \tilde d_\varepsilon(x)/2}
\bigg|\frac{b(t)- b(z)}{t- z} \bigg|^{(n+1)q}dt \\
\le \,\, &
\tilde d_\varepsilon(x) \max |b'(u)|^{(n+1)q},
\end{align*}
where the maximum is taken over $u\in [t,z]$
with $|t-x|\le \tilde d_\varepsilon(x)/2$
(by $[t,z]$ we denote the straight line segment
with the endpoints $t$ and $z$).
Note that for such $u$ we have $|{\rm Re}\, u-x| \le \tilde d_\varepsilon(x)/2$.
By Lemma 5.2, $|b'(u)| \le |b'({\rm Re}\, u)|$, and hence,
$$
\int_{|t-x|\le \tilde d_\varepsilon(x)/2}
\bigg|\frac{1-\overline{b(z)}b(t)}{t-\overline z} \bigg|^{(n+1)q}dt
\le
\tilde d_\varepsilon(x) \max_{|t-x|\le \tilde d_\varepsilon(x)/2} |b'(t)|^{(n+1)q}.
$$
According to Lemma 5.1, $|b'(t)|\le C_1 (\tilde d_\varepsilon(t))^{-1} \le
C_2(\tilde d_\varepsilon(x))^{-1}$ whenever $|t-x|<\tilde d_\varepsilon(x)/2$
which leads to the required estimate.

\end{proof}

\begin{Cor}\label{lem:levelsets3}
For each $\varepsilon \in (0,1)$ and $n\in \mathbb{N}$, there exists $C=C(\varepsilon,n)$
such that
$$
\|f^{(n)}\tilde d_\varepsilon^n\|_2\le C \|f\|_b, \qquad f \in \mathcal{H}(b).
$$
\end{Cor}

\begin{proof}
The statement follows immediately from Lemma
\ref{lem:levelsets2} and Theorem
\ref{maintheorem}.

\end{proof}

We conclude this section with a
a corollary of our Bernstein inequalities, concerning
the regularity on the boundary
for functions in $\mathcal{H}(b)$.
This technical result will be used later.

\begin{Cor}\label{cor:continuite-H(b)}
Let $I=[x_0,x_0+y_0]$ be a bounded interval on $\mathbb{R}$, $1< p<2$.
Assume that
\begin{eqnarray}\label{eq:integrale-spectre}
\int_I w_p(x)^{-2}dx<+\infty.
\end{eqnarray}
Then we have
\begin{enumerate}
\item[$\mathrm{a)}$] $]x_0,x_0+y_0[ \, \cap \, \sigma(b)=\emptyset$. In
particular, each function $f$ in $\mathcal{H}(b)$ is differentiable
on $]x_0,x_0+y_0[$.

\item[$\mathrm{b)}$]  $b$ is continuous on
the Carleson square $S(I)=[x_0,x_0+y_0]\times[0,y_0]$.
\end{enumerate}
\end{Cor}

\begin{proof} $\mathrm{a)}$ According to Theorem \ref{maintheorem}, there is a
constant $C>0$ such that
\[
\int_\mathbb{R} |f'(x)w_p(x)|^2\,dx\leq C\|f\|_b^2,\qquad f\in\mathcal{H}(b).
\]
Then, using (\ref{eq:integrale-spectre}) and the Cauchy--Schwartz
inequality, we get  $f'\in L^1(I)$ for any $f\in\mathcal{H}(b)$. Now choose
$z\in\mathbb{C}_+$ such that $b(z)\not=0$ and take $f=k_z^b$. We have
\[
f'(x)=-\overline{b(z)}\frac{b'(x)}{x-\overline
z}-\frac{k_z^b(x)}{x-\overline z}
\] and, since $k_z^b\in L^1(I)$, we conclude that
\begin{eqnarray}\label{eq:integrale-spectre2}
\int_{x_0}^{x_0+y_0}|b'(x)|\,dx<+\infty.
\end{eqnarray}
Now it follows immediately from the formula (\ref{eq:Rderivee}) for
$|b'(x)|$ that (\ref{eq:integrale-spectre2})
implies $]x_0,x_0+y_0[ \, \cap \, \sigma(b) = \emptyset$.
As a matter of  fact, this is obvious for the outer and the singular inner factors
since $\int_I (x-t)^{-2} dt =\infty$ for any $x\in I$; and if $b$ is
a Blaschke product with zeros $z_r$ tending to $x\in ]x_0,x_0+y_0[$,
then, for sufficiently large $r$,
$$
\int_{x_0}^{x_0+y_0} \frac {2 \,  \Ima z_r}{|x-z_r|^{2}}dx  \ge \pi,
$$
and so the integral in (\ref{eq:integrale-spectre2}) diverges.

$\mathrm{b)}$
By statement $\mathrm{a)}$,
$b$ is continuous on $S(I)$ except possibly at the points $x_0$
and $x_0+y_0$. It remains to show that $b$ is continuous at $x_0$ and
$x_0+y_0$. Fix $x_1\in ]x_0,x_0+y_0[$ and define
\[
b(x_0):=b(x_1)-\int_{x_0}^{x_1}b'(x)\,dx.
\]
(Note that this definition of $b(x_0)$ does not seem to correspond to the
classical one with non-tangential limits but, in fact, as we will see at the
end, they coincide). Since $b$ is differentiable on $]x_0,x_0+y_0[$, this
definition does
not depend on the choice of $x_1$ and we see from
(\ref{eq:integrale-spectre2}) that $b(x)$ tends to
$b(x_0)$ as $x\to x_0$ along $I$. Now let $z=x+iy\in S(I)$, with
$x\in [x_0,x_0+y_0/2[$, $y\in ]0,y_0/2[$. Write
$b(z)-b(x_0)=b(x+iy)-b(x+y)+b(x+y)-b(x_0)$. Using the continuity of
$b$ at $x_0$ along $I$, we have $b(x+y)-b(x_0)\to 0$, as $x\to x_0$
and $y\to 0$. Moreover, since $b$ is analytic on $\mathbb{C}_+ \cup \,
]x_0,x_0+y_0[$, we can write
\[
b(x+y)-b(x+iy)=(1-i)y\int_0^1 b'(t(x+y)+(1-t)(x+iy))\,dt.
\]
Applying Lemma \ref{Lem:derivee}, we get
\[
|b(x+y)-b(x+iy)|\leq\sqrt{2} \int_x^{x+y}|b'(u)|\,du.
\]
According to (\ref{eq:integrale-spectre2}), we
deduce that $b(x+y)-b(x+iy)\to 0$, as $x\to x_0$ and $y\to 0$.
Therefore, $b(z)\to b(x_0)$, as $z\to x_0$, $z\in S(I)$.
\end{proof}

\section{Carleson-type embedding theorems}
Weighted Bernstein-type inequalities of the
form (\ref{eq:intro}) turned out to be an
efficient tool for the study of the so-called
Carleson-type embedding theorems for the shift-coinvariant
subspaces $K_\Theta^p$. More precisely, given an inner
function $\Theta$, we want to describe the class of Borel
measure $\mu$ in the closed upper half-plane $\overline{\mathbb{C}_+}$
such that the embedding $K_\Theta^p\subset L^p(\mu)$ takes place. In other words, we are interested in the class of Borel
measure $\mu$ in $\overline{\mathbb{C}_+}$ such that there is a constant $C$ satisfying
\[
\|f\|_{L^p(\mu)}\leq C \|f\|_p,
\]
for all $f\in K_\Theta^p$. This problem was posed by Cohn
in \cite{Cohn82}. In spite of a number of beautiful results
(see, e.g., \cite{Cohn82, Cohn86, NV, VolbergTreil}),
the question still remains open in the general case.
Compactness of the embedding operator is also of interest
and is considered in \cite{CiMat03,Cohn-JOT86,Vo81}.

Methods based on the Bernstein-type inequalities allow to
give unified proofs and essentially generalize almost all known
results concerning these problems (see \cite{Baranov05, Baranov06}).
Here we obtain an embedding theorem for de Branges--Rovnyak spaces.
In the case of an inner function the first statement
coincides with a well-known theorem due to Volberg
and Treil \cite{VolbergTreil}.

A Carleson
 measure for the closed upper half-plane is called a {\it vanishing Carleson
 measure} if $\mu(S(x,h))/h \to 0$ whenever $h \to 0$
 or $\mbox{dist}\, (S(x,h), 0)\to \infty$.
 Vanishing Carleson measures in the closed unit disc are discussed, e.g.,
 in \cite{Power}. An equivalent definition for a vanishing Carleson measure $\nu$
 in the disc is that
\[
\int_{\overline{\mathbb{D}}} \frac{1-|z|^2}{|1-\overline z \zeta|^2}d\nu(\zeta)
\longrightarrow 0,\qquad \mbox{ as }\ |z|\to 1.
\]
 Changing the variables to the
 upper half-plane with $|w+i|^{-2}d\mu(w) = d\nu(\zeta)$, we obtain
\[
\int_{\overline{\mathbb{C}_+}} \frac{\Ima z}{|w-\overline z|^2} d\mu(w) \longrightarrow 0,
\]
 whenever either $\Ima z\to 0$ or $|z|\to+\infty$. It is easily seen
 that this condition is equivalent to the above definition of a
 vanishing Carleson measure. It is well known that an embedding $H^p(\mathbb{C}_+)\subset
 L^p(\mu)$ is compact if and only if $\mu$ is a vanishing Carleson
 measure.

\begin{Thm}\label{embeddings1}
Let $\mu$ be a Borel measure in $\overline{\mathbb{C}_+}$,
and let $\varepsilon\in (0,1)$.
\begin{enumerate}
\item[$\mathrm{(a)}$] Assume that $\mu(S(x,h)) \le Kh$ for all Carleson
squares
$S(x,h)$ satisfying
\[
S(x,h) \cap \widetilde{\Omega}(b, \varepsilon)
\ne\emptyset.
\]
Then
$\mathcal{H}(b)\subset L^2(\mu)$, that is, there is a constant $C>0$ such that
\[
\|f\|_{L^2(\mu)} \le C\|f\|_{b}, \qquad f\in \mathcal{H}(b).
\]
\item[$\mathrm{(b)}$] Assume that $\mu$ is a
vanishing Carleson measure for $\mathcal{H}(b)$, that is,
$\mu(S(x,h))/h\to 0$ whenever
$S(x,h)\cap \widetilde{\Omega}(b,\varepsilon)\ne\emptyset$
and $h\to 0$ or $\hbox{\rm{dist}}(S(x,h),0)\to +\infty$.
Then the embedding $\mathcal{H}(b)\subset L^2(\mu)$ is compact.
\end{enumerate}
\end{Thm}

In Theorem \ref{embeddings1} we need to verify the
Carleson condition only on a {\em special} subclass of squares.
Geometrically this means that when we are far from the spectrum
$\sigma(b)$, the measure $\mu$ in Theorem \ref{embeddings1}
can be essentially larger than
standard Carleson measures. The reason is that functions in $\mathcal{H}(b)$
have much more regularity at the points $x\in \mathbb{R}\setminus
\clos \sigma(b)$ where $|b(x)|=1$.
On the other hand, if $|b(x)|\le \delta <1$, almost everywhere on some interval
$I\subset \mathbb{R}$, then the functions in $\mathcal{H}(b)$ behave on $I$
essentially the same as a general element of $H^2(\mathbb{C}_+)$ on that interval,
and for any Carleson measure for $\mathcal{H}(b)$ its restriction
to the square $S(I)$ is a standard Carleson measure.

We will see that, for a class of functions $b$, the sufficient
condition of Theorem \ref{embeddings1} is also necessary. However,
it may be far from being necessary for certain functions $b$
even in the model space setting.

By {\it a closed square} in $\overline{\mathbb{C}}_+$,
we mean a set of the form \begin{equation}\label{def:square}
S(x_0,y_0,h):=\{x+iy:x_0\leq x\leq x_0+h,\,y_0\leq y\leq y_0+h\},
\end{equation}
where $x_0\in\mathbb{R}$, $y_0\geq 0$ and $h>0$;
by the {\it lower side} of the closed square $S(x_0,y_0,h)$ we mean the interval
$\{x+iy_0:x_0\leq x\leq x_0+h\}$.

We deduce Theorem \ref{embeddings1} from the following
more general result. Recall that
\[
w_p(z)=w_{p,1}(z)=
\min(\|(k_z^b)^2\|_q^{-p/(p+1)},
\|\rho^{1/q}\mathfrak{K}_{z,1}^\rho\|_q^{-p/(p+1)}).
\]

\begin{Thm}\label{embeddings2}
Let $\{S_k\}_{k\ge 1}$ be a sequence of closed squares in
$\overline{\mathbb{C_+}}$, let $I_k$ denote the lower side of the square
$S_k$, and let $\delta_{I_k}$
be the Lebesgue measure on $I_k$. Assume that the squares $S_k$
satisfy the following two conditions:
\begin{equation}
\label{cond1}
\sum\limits_k \delta_{I_k}\in \mathcal{C},
\end{equation}
and, for some $p$, $1<p<2$,
\begin{equation}
\label{cond2} \sup\limits_{k\ge 1,\,y\geq 0} |I_k| \, \int_{S_k\cap\{\Ima
z=y\}}w_p^{-2}(u)|du|<\infty.
\end{equation}
Let $\mu$ be a Borel
measure with $supp\,\mu \subset \bigcup\limits_k S_k$. Then
\begin{enumerate}
\item[$\mathrm{(a)}$] if $\mu(S_k)\le C|I_k|$, then $\mathcal{H}(b) \subset
L^2(\mu)$.
\item[$\mathrm{(b)}$] if, moreover,
$I_k\cap \clos \sigma(b)=\emptyset$, $k\geq 1$, and $\mu(S_k)=o(|I_k|)$,
$k\to\infty$, then
the embedding $\mathcal{H}(b) \subset L^2(\mu)$ is compact.
\end{enumerate}
\end{Thm}

For the model subspaces a result, analogous to Theorem \ref{embeddings2},
was obtained in \cite[Theorem 2.2]{Baranov05}.
For the sake of completeness, we include the proof.

\begin{proof} (a)  The idea of the proof is to replace
the measure $\mu$ with some Carleson measure $\nu$, and to estimate
the difference between the norms $\|f\|_{L^2(\mu)}$
and $\|f\|_{L^2(\nu)}$ using the Bernstein-type inequality of Section 4.

It follows from Corollary~\ref{cor:continuite-H(b)} (b)
that the set of functions $f\in \mathcal{H}(b)$
which are continuous on each of $S_k$
is dense in $\mathcal{H}(b)$ (take the reproducing kernels $k_z^b$,
$z\in \mathbb{C}^+$). Thus it is sufficient to prove the estimate
$\|f\|_{L^2(\mu)} \le C\|f\|_b$
only for $f \in \mathcal{H}(b)$ continuous on $\bigcup\limits_k S_k$.
Now let $f \in \mathcal{H}(b)$ be continuous on each of $S_k$.
Then there exist $w_k \in S_k$ such that
\begin{equation}
\label{61-4}
\|f\|^2_{L^2(\mu)} \le \sum\limits_k |f(w_k)|^2\mu(S_k) \le
\sup\limits_k \frac{\mu(S_k)}{|I_k|} \cdot \sum\limits_k |f(w_k)|^2 |I_k|.
\end{equation}
Statement (a) will be proved as soon as we show that
\begin{equation}
\label{61-5}
\sum\limits_k |f(w_k)|^2 |I_k| \le C\|f\|_b^2
\end{equation}
where the constant $C$ does not depend on $f$
and on the choice of $w_k\in S_k$.

Consider the intervals $J_k = S_k\cap \{\Ima z= \Ima w_k\}$.
Let $\nu =\sum_k \delta_{J_k}$. Then it follows from (\ref{cond1})
that $\nu\in \mathcal{C}$ (and the Carleson constants $C_\nu$ of
such measures $\nu$ are uniformly bounded). We have
\begin{equation}
\label{61-3}
\bigg(\sum\limits_k |f(w_k)|^2 |I_k|\bigg)^{1/2}
\le \|f\|_{L^2(\nu)} +
\bigg(\sum\limits_k \int_{J_k} |f(z)-f(w_k)|^2 |dz|\bigg)^{1/2},
\end{equation}
and $\|f\|_{L^2(\nu)} \le C_1 \|f\|_2\le C_1\|f\|_b$.

We estimate the last term in (\ref{61-3}).
For $z\in J_k$ denote by $[z, w_k]$ the straight line interval
with the endpoints $z$ and $w_k$.
Then $f(z)-f(w_k) = \int_{[z, w_k]} f'(u) du$
(in the case $J_k \subset \mathbb{R}$ note that, by Corollary~\ref{cor:continuite-H(b)} (a),
any $f\in \mathcal{H}(b)$ is differentiable on $J_k$
except, may be, at the endpoints). So, by the
Cauchy--Schwartz inequality,
$$
\sum\limits_k \int_{J_k} |f(z)-f(w_k)|^2 |dz| \le
\sum\limits_k \int_{J_k} \bigg|
\int_{J_k} |f'(u)| |du|\bigg|^2 |dz|
$$
$$
\le  \sum\limits_k |J_k|
\bigg(\int_{J_k} w_p^{-2}(u) |du|\bigg)
\bigg(\int_{J_k} |f'(u)|^2 w_p^2(u) |du|\bigg).
$$
By (\ref{cond2}), we obtain
$$
\sum\limits_k \int_{J_k} |f(z)-f(w_k)|^2 |dz|
\le C_2 \sum\limits_k \int_{J_k} |f'(u)|^2 w_p^2(u) |du|
$$
$$
= C_2 \|f' w_p\|^2_{L^2(\nu)} \le C_3\|f\|_b^2,
$$
where the last inequality follows from Theorem \ref{maintheorem}.
\smallskip

(b) For a Borel set
$E\subset\overline{\mathbb{C}_+}$ define the operator
${\mathcal I}_E:  \mathcal{H}(b)\to L^2(\mu)$ by
${\mathcal I}_E f= \chi_E f$
where $\chi_E$ is the characteristic function of $E$.
For $N\in \mathbb{N}$ put $F_N= \bigcup\limits_{k=1}^N S_k$
and $\widehat F_N = \overline{\mathbb{C}_+} \setminus F_N$.
As above we assume that
$f\in \mathcal{H}(b)$ is continuous on $\bigcup\limits_k S_k$.
Then it follows from (\ref{61-4}) and (\ref{61-5}) that
$$
\int_{\widehat F_N} |f|^2 d\mu \le C\sup\limits_{k>N}
\frac{\mu(S_k)}{|I_k|}\|f\|_b^2,
$$
and so $\|{\mathcal I}_{\widehat F_N}\| \to 0$,
$N\to\infty$. Statement (b) will be proved as soon as we show that
${\mathcal I}_{F_N}$ is a compact operator for any $N$ (thus,
our embedding operator
${\mathcal I}_{F_N}+{\mathcal I}_{\widehat F_N}$
may be approximated in the operator norm by compact operators
${\mathcal I}_{F_N}$). Clearly, it suffices to prove the compactness
of ${\mathcal I}_{S_k}$ for each fixed $k$.

We approximate ${\mathcal I}_{S_k}$ by finite rank operators.
For a given $\epsilon>0$, partition the square $S_k$ into finite union of
squares $\{\tilde S_l\}_{l=1}^L$ with
pairwise disjoint interiors so that
\begin{equation}
\label{61-8}
\bigg(\, \int_{[\zeta, z]}
w^{-2}_p(u) |du| \bigg)<\epsilon
\end{equation}
for any $l$, $1\le l\le L$, and any  $\zeta, z \in \tilde S_l$.
Such a partition exists since $I_k\cap \clos \sigma(b)=\emptyset$,
$k\geq 1$. Indeed, $b$ is analytic in a neighborhood
of $S_k$, and the norms involved in the definition of $w_p(z)$ are
continuous on $S_k$.

Now fix $\zeta_l \in \tilde S_l$ and
consider the finite rank operator $T: \mathcal{H}(b) \to L^2(\mu)$,
$(Tf)(z) = \sum_{l=1}^L f(\zeta_l) \chi_{\tilde S_l}(z)$.
We show that $\| {\mathcal I}_{S_k} - T\|^2 \le C\epsilon$. As in the proof
of (a), we have
$$
\|({\mathcal I}_{S_K} - T) f\|_{L^2(\mu)}^2 =
\sum\limits_{l=1}^L \int_{\tilde S_l}|f(z)-f(\zeta_l)|^2d\mu(z)
$$
$$
\le
\sum\limits_{l=1}^L   \int_{\tilde S_l}
\bigg(\, \int_{[\zeta_l, z]} |f'(u)|^2 w_p^2(u) |du| \bigg)\cdot
\bigg(\, \int_{[\zeta_l, z]} w^{-2}_p(u) |du|
\bigg) d\mu(z).
$$
By Theorem \ref{maintheorem},
$$
 \int_{[\zeta_l, z]} |f'(u)|^2 w^2_p(u) |du| \le C_1 \|f\|_b^2
$$
where $C_1$ does not depend on $f\in \mathcal{H}(b)$,
$1\le l\le L$ and $z\in \tilde S_l$.
Hence, by (\ref{61-8}),
$$
\|({\mathcal I}_{S_K} - T) f \|_{L^2(\mu)}^2 \le C_1\epsilon
\|f\|_b^2
\sum\limits_{l=1}^L \mu(\tilde S_l) = C_1\epsilon \mu(S_k) \|f\|_b^2.
$$
We conclude that ${\mathcal I}_{S_K}$ may be approximated
by finite rank operators and is, therefore, compact.
\end{proof}

We comment now on a couple of details of the proof
where the situation differs from the inner case.

\begin{remark}{\rm In the inner case $b=\Theta$
one can prove the estimate $\|f\|_{L^2(\mu)}\leq C\|f\|_2$
for functions $f$ in $K_\Theta^2$ which are continuous on
the closed upper half-plane $\overline{\mathbb{C}_+}$ and then
use a result of Aleksandrov
\cite{aleksandrov} which says that such functions are dense
in $K_\Theta^2$. We do not know if this result is still valid
in $\mathcal{H}(b)$. To avoid this difficulty,
in the proof of Theorem \ref{embeddings2}, we
used the density in $\mathcal{H}(b)$ of the
functions continuous on all squares $S_k$. }
\end{remark}

\begin{question}\label{quest-1}
{\rm Let $b$ be in the unit ball of
$H^\infty(\mathbb{C}_+)$.
Is it true that the set of functions $f$ in $\mathcal{H}(b)$,
continuous on $\overline{\mathbb{C}_+}$, is dense in $\mathcal{H}(b)$? }
\end{question}

\begin{remark}{\rm
In the inner case, in Theorem \ref{embeddings2}, the assumption
(\ref{cond2}) can be replaced by the weaker assumption
(only for the lower side of the square)
\begin{eqnarray}\label{eq:hypo-interieur}
 \sup\limits_{k\geq 1} |I_k| \, \int_{I_k}w_p^{-2}(u)|du|<\infty.
\end{eqnarray}
It was noticed in \cite[Corollary 4.7]{Baranov05} that in the inner case,
for $q>1$, there exists $C=C(q)>0$ such that,
for any $x\in\mathbb{R}$ and $0\leq
y_2\leq y_1$, we have
\begin{eqnarray}\label{eq:monotonie-norme}
\|k_{x+iy_1}^b\|_q\leq C(q) \|k_{x+iy_2}^b\|_q.
\end{eqnarray}
Thus, it follows from (\ref{eq:monotonie-norme}) that if
the sequence $\{S_k\}$ satisfies (\ref{eq:hypo-interieur}),
then it also satisfies (\ref{cond2}). }
\end{remark}

\begin{question}\label{quest-2}
{\rm Does the monotonicity property (\ref{eq:monotonie-norme}) of the
norms of the reproducing kernels along the rays parallel to imaginary axis
remains true for a general $b$? (It is true for $q=2$,
but this is not the interesting case for us.) }
\end{question}

\begin{proof} {\it of Theorem \ref{embeddings1}.} $\mathrm{(a)}$ Consider the
open set $E = \mathbb{R}\setminus
\clos \widetilde{\Omega}(b, \varepsilon)$. If $E=\emptyset$, then $\mu$ is a
Carleson measure and $\mathcal{H}(b)\subset H^2(\mathbb{C}_+)\subset L^2(\mu)$. So we may
assume that $E\ne\emptyset$ and we can write it as a union
of disjoint intervals $\Delta_l$. Note that
$\int_{\Delta_l}(\tilde d_\varepsilon(t))^{-1} dt=\infty $.
Hence, partitioning the intervals $\Delta_l$,
we may represent $E$ as a union of intervals
$I_k$ with mutually disjoint interiors such that
$$
\int_{I_k} \big[ \tilde d_\varepsilon(t) \big]^{-1} dt =\frac{1}{2}.
$$
It follows that there exists $x_k\in I_k$ such that
$\tilde d_\varepsilon(x_k) =2|I_k|$. Hence, for any $x\in I_k$,
$\tilde d_\varepsilon(x)\ge \tilde d_\varepsilon(x_k) - |I_k| =|I_k|$ and
$\tilde d_\varepsilon(x) \le 3|I_k|$.
This implies
$$
|I_k| \,\,  \int_{I_k} \big[ \tilde d_\varepsilon(t) \big]^{-2}
dt \le 1,
$$
and using Lemma \ref{lem:levelsets2}, we conclude that the intervals $I_k$
satisfy (\ref{cond2}). Condition (\ref{cond1}) is obvious.

Let $S_k=S(I_k)$ be the Carleson square with the lower side $I_k$,
let $F=\bigcup_k S_k$, and let $G=\overline {\mathbb{C_+}}\setminus
F$. Put $\mu_1=\mu|_F$ and $\mu_2=\mu|_G$. We show that
the measure $\mu_1$ satisfies the conditions of Theorem
\ref{embeddings2} whereas $\mu_2$ is a usual Carleson measure
(and, thus, $\mathcal{H}(b)\subset H^2(\mathbb{C}_+)\subset L^2(\mu_2)$).

Let us show that $\mu_1(S_k) \le C_2|I_k|$. Indeed,
it follows from the estimate $|I_k| \le
\tilde d_\varepsilon(x)\le 3|I_k|$, $x\in I_k$,
that $S(6 I_k)\cap \widetilde{\Omega}(b, \varepsilon) \ne \emptyset$
(by $6I_k$ we denote the 6 times larger interval
with the same center as $I_k$).
By the hypothesis, $\mu_1 (S_k) \le \mu (S(6I_k)) \le C |I_k|$.
Hence, $\mu_1$ satisfies the conditions of Theorem \ref{embeddings2} (a),
and so $\mathcal{H}(b) \subset L^2(\mu_1)$.

Now we show that $\mu_2 \in \mathcal{C}$.
Assume that $S(I) \cap G\ne\emptyset$
for some interval $I\subset \mathbb{R}$, and let $z=x+iy\in S(I) \cap G$.
If $x\in \clos \widetilde{\Omega}(b, \varepsilon)$,
then $S(2I)\cap\widetilde{\Omega}(b, \varepsilon)\ne \emptyset$.
Otherwise, if $x\in I_k$ for some $k$, then
$\tilde d_\varepsilon(x) \le 3|I_k|\le 3|I|$ since
$z\in S(I)\setminus S(I_k)$. Thus
\begin{eqnarray}\label{eq:levint}
S(6I)\cap\widetilde{\Omega}(b,
\varepsilon)\ne \emptyset.
\end{eqnarray}
By the hypothesis, $\mu_2(S(I)) \le \mu(S(6I)) \le C|I|$,
and so $\mu_2$ is a Carleson measure.

$\mathrm{(b)}$ Let $F,G, \mu_1$ and $\mu_2$
be the same as above. We show that $\mu_1$
satisfies the conditions of Theorem \ref{embeddings2} (b), whereas
$\mu_2$ is a vanishing Carleson measure. Indeed, we can split
the family $\{S_k\}$ into two families
$\{S_k\}_{k\in K_1}$ and $\{S_k\}_{k\in K_2}$ such that
$|I_k|\to 0$, $k\to\infty$, $k\in K_1$, whereas
${\rm dist}\, (I_k, 0)\to\infty$ when $k\to\infty$, $k\in K_2$.
Since $S(6 I_k)\cap \widetilde{\Omega}(b, \varepsilon) \ne \emptyset$
we conclude that Theorem \ref{embeddings2} (b) applies to
$\mu_1$ and the embedding $\mathcal{H}(b)\subset L^2(\mu_1)$
is compact. Finally, any Carleson square $S(I)$
with $S(I) \cap G\ne\emptyset$ satisfies (\ref{eq:levint}), and so, by the assumptions
of Theorem \ref{embeddings1} (b),
$\mu_2$ is a vanishing Carleson measure.

\end{proof}

We state an analogous result for the spaces
in the unit disc (for the case of inner functions statement (b) is proved in
\cite{Baranov06}; it answers a question posed in \cite{CiMat03}).

\begin{Thm}\label{embeddings3}
Let $\mu$ be a Borel measure in the closed unit disc $\overline{\mathbb{D}}$, and
let $\varepsilon\in (0,1)$.
\begin{enumerate}
\item[$\mathrm{(a)}$] Assume that $\mu(S(x,h) \le Ch$ for all Carleson squares
$S(x,h)$ such that $S(x,h) \cap \widetilde{\Omega}(b, \varepsilon)
\ne\emptyset$. Then $\mathcal{H}(b)\subset L^2(\mu)$.
\item[$\mathrm{(b)}$] If, moreover, $\mu(S(x,h))/h \to 0$ when $h\to 0$
and $S(x,h) \cap \widetilde{\Omega}(b, \varepsilon) \ne\emptyset$, then
the embedding $\mathcal{H}(b)\subset L^2(\mu)$
is compact.
\end{enumerate}
\end{Thm}

For a class of functions $b$ the converse to Theorem \ref{embeddings1}
is also true. We say that $b$ satisfies the {\it connected level set
condition}
if the set $\Omega(b, \varepsilon)$ is connected for some $\varepsilon \in (0,1)$.
Our next result is analogous to certain results from \cite{Cohn82} and
to \cite[Theorem 3]{VolbergTreil}.

\begin{Thm}\label{embeddings4}
Let $b$ satisfy the connected level set condition for some $\varepsilon\in (0,1)$.
Assume that $\Omega(b, \varepsilon)$ is unbounded and
$\sigma(b) \subset \clos \Omega(b, \varepsilon)$. Let  $\mu$ be a Borel measure
on $\overline{\mathbb{C}_+}$. Then the following statements are equivalent:
\begin{enumerate}
\item[$\mathrm{(a)}$] $\mathcal{H}(b)\subset L^2(\mu)$.
\item[$\mathrm{(b)}$] There exists $C>0$ such that
$\mu(S(x,h)) \le Ch$ for all Carleson squares
$S(x,h)$ such that $S(x,h) \cap \widetilde{\Omega}(b, \varepsilon)
\ne\emptyset$.
\item[$\mathrm{(c)}$] There exists $C>0$ such that
\begin{equation}
\label{poisson} \int_{\overline{\mathbb{C}}_+} \frac{\Ima z}{|\zeta-\overline
z|^2}d\mu(\zeta)
\le  \frac{C}{1-|b(z)|}, \qquad z\in \mathbb{C}_+.
\end{equation}
\end{enumerate}
\end{Thm}

\begin{proof}
The implication $\mathrm{(b)} \Longrightarrow \mathrm{(a)} $ holds for any
$b$ by Theorem
\ref{embeddings1}, and the
implication $\mathrm{(a)}  \Longrightarrow \mathrm{(c)} $ is trivial (apply the inequality
$\|f\|_{L^2(\mu)}\le C\|f\|_b$ to $f=k_z^b$). To prove that
$\mathrm{(c)}  \Longrightarrow \mathrm{(b)} $,
we use an argument from \cite{VolbergTreil}.
Let $S(x,h)$ be a Carleson square such that
$S(x,h) \cap \widetilde{\Omega}(b, \varepsilon) \ne\emptyset$. Since
$\sigma(b) \subset \clos \Omega(b, \varepsilon)$ it follows that
$S(x,2h) \cap \Omega(b, \varepsilon) \ne\emptyset$. Choose $z_1\in S(x,
2h)\cap\mathbb{C}_+$
with $|b(z_1)|<\varepsilon$. Now consider
$S(x,3h)$. Since $\Omega(b, \varepsilon)$ is connected and unbounded,
there exists a point $z_2$ on the boundary
of $S(x,3h)$ such that $|b(z_2)|<\varepsilon$. Hence, there exists a continuous
curve
$\gamma$ connecting $z_1$ and $z_2$ and such that $|b|<\varepsilon$ on $\gamma$.
Now let $z=x+ih$.
Applying the theorem on two constants to the domain ${\rm Int}\, S(x,3h)
\setminus \gamma$ we conclude that $|b(z)|\le \delta$ where $\delta\in
(0,1)$ depends
only on $\varepsilon$. Then inequality (\ref{poisson}) implies
$$
h \int_{S(x,h)} \frac{d\mu(\zeta)} {|\zeta-\overline z|^2}
\le C(1-\delta)^{-1}.
$$
It remains to note that $|\zeta-\overline z|\le
C_1 h$, $\zeta\in S(x,h)$ to obtain $\mu(S(x,h)) \le C_2 h$.

\end{proof}

\begin{example}
{\rm Examples are known of inner functions
satisfying the connected level set condition.
We would like to emphasize that there are also many outer functions
satisfying the conditions of Theorem \ref{embeddings4}.
For example, let $b(z)=\exp(\frac{i}{\pi}\log z)$, where $\log z$
is the main branch of the logarithm in $\mathbb{C} \setminus (-\infty,0]$.}
\end{example}

\begin{remark}
{\rm We see that if $b$ satisfies the conditions of Theorem \ref{embeddings4},
then it suffices to verify the inequality $\|f\|_{L^2(\mu)}\leq C\|f\|_b$
for the reproducing kernels of the space $\mathcal{H}(b)$ to get it for
all functions $f$ in $\mathcal{H}(b)$. Recently, Nazarov and Volberg
\cite{NV} showed that it is no longer true in the general case. }
\end{remark}

\section{Stability of bases of reproducing kernels}
Another application of Bernstein inequalities for model subspaces $K_\Theta^p$
is considered in \cite{Baranov05a};
it is connected with stability of Riesz bases and frames
of reproducing kernels $(k^\Theta_{\lambda_n})$ under small perturbations
of the points $\lambda_n$. Riesz bases of reproducing kernels in de
Branges--Rovnyak spaces $\mathcal{H}(b)$ were studied in \cite{Fricain}.
Making use of Theorem \ref{maintheorem} we extend the results of
\cite{Baranov05a} to the spaces $\mathcal{H}(b)$.

For $\lambda\in\mathbb{C}_+\cup E_2(b)$, we denote by $\kappa_\lambda^b$ the
normalized reproducing kernel at the point $\lambda$, that is,
$\kappa_\lambda^b= k_\lambda^b/(2\pi i \, \|k_\lambda^b\|_b)$. Let
$(\kappa_{\lambda_n}^b)_{n\geq 1}$ be a Riesz basis in $\mathcal{H}(b)$, let
$\lambda_n \in G_n$ and let $G=\bigcup_n G_n\subset
\overline{\mathbb{C}_+}$ satisfy the following properties.
\begin{enumerate}
\item[$\mathrm{(i)}$] There exist positive constants $c$ and $C$ such that
\[
c\leq \frac{\|k_{z_n}^b\|_b}{\|k_{\lambda_n}^b\|_b}\leq C,\qquad
z_n\in G_n.
\]
\item[$\mathrm{(ii)}$] For any $z_n\in G_n$, the measure
$\nu=\sum_n \delta_{[\lambda_n,z_n]}$
is a Carleson measure and, moreover, the Carleson constants $C_\nu$ of such
measures (see (\ref{carl})) are uniformly bounded with respect to $z_n$.
Here $[\lambda_n,z_n]$ is the straight line interval with the endpoints
$\lambda_n$ and $z_n$, and $\delta_{[\lambda_n,z_n]}$ is the Lebesgue measure
on the interval.
\end{enumerate}

\begin{remark}\label{Rem1:stabilite}
{\rm As in the inner case, it should be noted that for $\lambda_n\in\mathbb{C}_+$,
there always exist non-trivial sets $G_n$ satisfying $(i)$ and $(ii)$. More
precisely, we can take
\[
G_n:=\{z\in\mathbb{C}_+:|z-\lambda_n|<r\Ima\lambda_n\},
\]
for sufficiently small $r>0$. Indeed,  we know \cite{Fricain} that if
$(\kappa_{\lambda_n}^b)_{n\geq 1}$ is a Riesz basis in $\mathcal{H}(b)$, then
$(\lambda_n)_{n\great 1}$ is a Carleson sequence, that is,
\[
\inf_{k\great
1}\prod_{n\not=k}\left|\frac{\lambda_n-\lambda_k}{\lambda_n-\overline\lambda_k}\right|>0.
\]
In particular, the measure
$\nu:=\sum_n\Ima \lambda_n\, \delta_{\lambda_n}$ is a Carleson measure.
Therefore, we see that $G_n$ satisfy $\mathrm{(ii)}$. Moreover,
using Lemma \ref{Lem:stabilite} below, we see that $G_n$ satisfy
also the condition $\mathrm{(i)}$.}
\end{remark}

Recall that
$w_p(z)=\min(\|(k_z^b)^2\|_q^{-p/(p+1)},\|\rho^{1/q}
\mathfrak{K}_{z,1}^\rho\|_q^{-p/(p+1)})$.

\begin{Thm}\label{thm:stabilite-base}
Let $(\lambda_n)_{n\geq 1}\subset \mathbb{C}_+\cup E_2(b)$ be such that
$(\kappa_{\lambda_n}^b)_{n\geq 1}$ is a Riesz basis in $\mathcal{H}(b)$ and let $p\in
[1,2)$. Then for any set $G=\bigcup_n G_n$ satisfying $\mathrm{(i)}$ and
$\mathrm{(ii)}$, there is $\varepsilon>0$ such that the system of reproducing
kernels $(\kappa_{\mu_n}^b)_{n\geq 1}$ is a Riesz basis whenever $\mu_n\in
G_n$ and
\begin{eqnarray}\label{eq:stabilite-base}
\sup_{n\geq
1}\frac{1}{\|k_{\lambda_n}^b\|_b^2}\int_{[\lambda_n,\mu_n]}w_p(z)^{-2}|dz|<\varepsilon.
\end{eqnarray}
\end{Thm}

\begin{proof}
Since $\mu_n\in G_n$, the condition $\mathrm{(i)}$ implies that
$\|k_{\mu_n}^b\|_b\asymp \|k_{\lambda_n}^b\|_b$ and thus
$(\kappa_{\mu_n}^b)_{n\geq 1}$ is a Riesz basis if and only if
$(\widetilde\kappa_{\mu_n}^b)_{n\geq 1}$ is a Riesz basis where
\[
\widetilde\kappa_{\mu_n}^b=\dfrac{k_{\mu_n}^b}{2\pi i \, \|k_{\lambda_n}^b\|_b}.
\]
In view of \cite[Lemma 2.3]{Baranov05a}, it suffices to check the estimate
\begin{eqnarray}\label{eq:cle-stabilite}
\sum_{n=1}^\infty|\langle
f,\kappa_{\lambda_n}^b-\widetilde\kappa_{\mu_n}^b\rangle_b|^2\leq \varepsilon
\|f\|_b^2,\qquad f\in\mathcal{H}(b),
\end{eqnarray}
for sufficiently small $\varepsilon>0$. Now it follows from
(\ref{eq:stabilite-base}) and Corollary \ref{cor:continuite-H(b)} (a)
that any $f$ in $\mathcal{H}(b)$ is differentiable in $]\lambda_n,\mu_n[$.
Moreover, the set of functions in $\mathcal{H}(b)$ which are continuous on
$[\lambda_n,\mu_n]$ is dense in $\mathcal{H}(b)$ (take the set of
reproducing kernels). Therefore, we can prove
(\ref{eq:cle-stabilite}) only for functions $f\in\mathcal{H}(b)$ continuous on  $[\lambda_n,\mu_n]$. Then
\[
|\langle
f,\kappa_{\lambda_n}^b-\widetilde\kappa_{\mu_n}^b\rangle_b|^2
=\dfrac{|f(\lambda_n)-f(\mu_n)|^2}{\|k_{\lambda_n}^b\|_b^2}
=\dfrac{1}{\|k_{\lambda_n}^b\|_b^2}\left|\int_{[\lambda_n,\mu_n]}f'(z)\,dz\right|^2.
\]
By the Cauchy--Schwartz inequality and (\ref{eq:stabilite-base}), we get
\[
|\langle f,\kappa_{\lambda_n}^b-\widetilde\kappa_{\mu_n}^b\rangle_b|^2\leq
\varepsilon \int_{[\lambda_n,\mu_n]}|f'(z)w_p(z)|^2|dz|.
\]
It follows from assumption $\mathrm{(ii)}$ that
$\nu:=\sum_n\delta_{[\lambda_n,\mu_n]}$ is a Carleson measure with a constant
$C_\nu$ which does not exceed some absolute constant depending only on $G$.
Hence, according to Theorem \ref{maintheorem}, we have
\begin{align*}
\sum_{n=1}^\infty|\langle
f,\kappa_{\lambda_n}^b-\widetilde\kappa_{\mu_n}^b\rangle_b|^2\leq \,\,& \varepsilon
\sum_{n=1}^\infty \int_{[\lambda_n,\mu_n]}|f'(z)w_p(z)|^2|dz|\\
=\,\,& \varepsilon \|f'w_p\|_{L^2(\nu)}^2 \leq  C \, \varepsilon \, \|f\|_b^2,
\end{align*}
for a constant $C$ which depends on $G$, $(\lambda_n)$ and $p$. Then Lemma
2.3 of \cite{Baranov05a} implies that we can choose a sufficiently small
$\varepsilon>0$ such that $(\widetilde\kappa_{\mu_n}^b)_{n\geq 1}$ is a Riesz basis
in $\mathcal{H}(b)$.

\end{proof}

Denote by $\rho(z,\omega)$ the pseudohyperbolic distance between
$z$ and $\omega$,
\[
\rho(z,\omega) := \left|\frac{z-\omega}{z-\overline\omega}\right|.
\]
For the proof of the next corollary we need the following
well-known property.

\begin{Lem}\label{Lem:stabilite}
Let $b \in H^\infty(\mathbb{C}_+)$ with $\|b\|_\infty \leq 1$ and
$\varepsilon_0\in (0,1)$. Then there exist constants $C_1,C_2>0$
(depending only on $\varepsilon_0$) such that for any
$z,\omega\in\mathbb{C}_+$ satisfying $\rho(z,\omega)<\varepsilon_0$, we have
\begin{eqnarray}\label{eq1:stabilite}
C_1\leq \frac{1-|b(z)|}{1-|b(\omega)|}\leq C_2.
\end{eqnarray}

\end{Lem}

\begin{proof}
For the case of an inner function, the proof can be found, e.g., in
\cite[Lemma 4.1]{Baranov05a}. Since for
$0\leq t_1,t_2,s_1,s_2<1$, we have
\[
\frac{1-t_1t_2}{1-s_1s_2}\leq \frac{1-t_1}{1-s_1}+\frac{1-t_2}{1-s_2},
\]
the inner and outer factors of $b$ can be treated separately and we can
assume that $b$ is outer.
It follows easily from $\rho(z,\omega)<\varepsilon_0$ that
\begin{equation}\label{eq2:stabilite}
|z-\omega|<\frac{2\varepsilon_0}{1-\varepsilon_0}\Ima\omega
\end{equation}
and
\[
\frac{1-\varepsilon_0}{1+\varepsilon_0}<\frac{\Ima
z}{\Ima\omega}<\frac{1+\varepsilon_0}{1-\varepsilon_0}.
\]
Hence
$$
\frac{\Ima z}{\pi}\int_\mathbb{R} \frac{\big|\log|b(t)|\big|}{|t-z|^2}\,dt
\asymp \frac{\Ima \omega}{\pi}\int_\mathbb{R}
\frac{\big|\log|b(t)|\big|}{|t-\omega|^2}\,dt.
$$
Since $b$ is outer, we have
\begin{eqnarray}\label{eq:lem-pseudohyperbolique}
\log |b(z)|= - \frac{\Ima z}{\pi}\int_\mathbb{R}
\frac{\big|\log|b(t)|\big|}{|t-z|^2}\,dt \asymp \log|b(\omega)|,
\end{eqnarray}
which implies $1-|b(z)|\asymp 1-|b(\omega)|$.

\end{proof}

\begin{Cor}\label{Cor:stabilite}
Let $(\lambda_n) \subset\mathbb{C}_+$, let $(\kappa_{\lambda_n}^b)_{n\geq 1}$ be a
Riesz basis in $\mathcal{H}(b)$, and let $\gamma>1/3$. Then there is $\varepsilon>0$ such
that the system $(\kappa_{\mu_n}^b)_{n\geq 1}$ is a Riesz basis whenever
\begin{eqnarray}\label{eqcle:stabilite}
\left|\dfrac{\lambda_n-\mu_n}{\lambda_n-\overline{\mu_n}}\right|\leq \varepsilon
(1-|b(\lambda_n)|)^\gamma.
\end{eqnarray}
\end{Cor}

\begin{proof}
By Remark \ref{Rem1:stabilite}, for sufficiently small $r>0$, the sets
$G_n=\{z:|z-\lambda_n|\leq r\Ima\lambda_n\}$
satisfy the conditions $\mathrm{(i)}$ and $\mathrm{(ii)}$. Let
$(\mu_n)_{n\geq 1}$ satisfy (\ref{eqcle:stabilite}). Then, by
(\ref{eq2:stabilite}), we have
\begin{eqnarray}\label{eq4:stabilite}
|\lambda_n-\mu_n|\leq
\frac{2\varepsilon}{1-\varepsilon}(1-|b(\lambda_n)|)^\gamma\Ima\lambda_n.
\end{eqnarray}
Therefore, if $\varepsilon$ is sufficiently small, then $\mu_n\in G_n$.
Without loss of generality, we can assume that $\gamma<1$ and since
$\gamma>1/3$, there exists $1<p<2$ such that
$2\frac{p-1}{p+1}=1-\gamma$.
Let $q$ be the conjugate exponent of $p$ and note that
$\frac{2p}{q(p+1)}=1-\gamma$.

Then it follows from Lemma \ref{lem:minoration} that there is a constant
$C=C(p)>0$ such that
\[
w_p(z)\geq C \frac{\Ima z}{(1-|b(z)|)^{\frac{p}{q(p+1)}}},\qquad z\in\mathbb{C}_+.
\]
Therefore, by Lemma \ref{Lem:stabilite}, we have
\[
w_p^{-2}(z)\leq C_1\frac{(1-|b(\lambda_n)|)^{1-\gamma}}{(\Ima \lambda_n)^{2}}
\]
for $z\in [\lambda_n,\mu_n]$. Hence,
\[
\frac{1}{\|k_{\lambda_n}^b\|_b^2}\int_{[\lambda_n,\mu_n]}w_p(z)^{-2}|dz|\leq
C_2\frac{\Ima\lambda_n}{1-|b(\lambda_n)|}|\lambda_n-\mu_n|\frac{(1-|b(\lambda_n)|)^{1-\gamma}}
{(\Ima\lambda_n)^2}
\]
and using (\ref{eq4:stabilite}), we obtain
\[
\frac{1}{\|k_{\lambda_n}^b\|_b^2}\int_{[\lambda_n,\mu_n]}w_p(z)^{-2}|dz|\leq
C_3\varepsilon.
\]
To complete the proof, take a sufficiently small $\varepsilon$ and
apply Theorem \ref{thm:stabilite-base}.

\end{proof}

\begin{remark}
{\rm It should be noted that all the statements remain valid if we are interested
in the stability of Riesz sequences of reproducing kernels, that is, of
systems of reproducing kernels which constitute Riesz bases in their closed
linear spans. }
\end{remark}

\begin{remark}
{\rm
In the case where
\begin{eqnarray}\label{eq:critere-RBN}
\sup_{n\geq 1}|b(\lambda_n)|<1,
\end{eqnarray}
the stability condition (\ref{eqcle:stabilite}) is equivalent to
\[
\left|\dfrac{\lambda_n-\mu_n}{\lambda_n-\overline{\mu_n}}\right|\leq \varepsilon,
\]
and we essentially get the result of stability obtained in the inner
case in \cite{Fricain-JOT}. Moreover, if $b$ is an extreme point of
the unit ball of $H^\infty(\mathbb{C}_+)$ and if (\ref{eq:critere-RBN})
is satisfied, then a criterion for $(\kappa_{\lambda_n}^n)$ to be a Riesz
basis of $\mathcal{H}(b)$ is given in \cite{Fricain}.
On the other hand, in the non-extreme case, there are no
Riesz bases of $\mathcal{H}(b)$ and the previous results
(Theorem \ref{thm:stabilite-base} and Corollary \ref{Cor:stabilite})
apply only for Riesz sequences. }
\end{remark}

\end{document}